\documentclass[10pt,a4paper]{amsart}
\usepackage{amsaddr}
\usepackage{amssymb,amsthm,amsmath,amsfonts,amscd}
\usepackage{graphicx}
\usepackage{url}
\usepackage{color}
\usepackage[active]{srcltx}
\usepackage[matrix,arrow]{xy}
\usepackage{mathrsfs}
\usepackage{enumerate}
\usepackage{amsopn} 
\usepackage{bbm} 
\usepackage{ulem}
\usepackage{cite}
\usepackage{hyperref}
\allowdisplaybreaks[1]
\usepackage[english]{babel}
\usepackage{bbm}
\usepackage{stmaryrd}

\newcommand{\E}{\mathbb{E}}

\newcounter{cprop}[section]

\newcommand{\mcB}{\mathcal{B}}

\newtheorem{definition}[cprop]{Definition}
\newtheorem{remark}[cprop]{Remark}
\newtheorem{lemma}[cprop]{Lemma}
\newtheorem{example}[cprop]{Example}
\newtheorem{proposition}[cprop]{Proposition}
\newtheorem{theorem}[cprop]{Theorem}
\newtheorem{corollary}[cprop]{Corollary}

\usepackage{mdef}

\title[Random attractors for degenerate SPDE]{Random attractors for degenerate stochastic partial differential equations}
\author[B. Gess]{Benjamin Gess}
\email{gess@math.tu-berlin.de}
\address{Institut f\"ur Mathematik, Technische Universit\"at Berlin (MA 7-5)\\
Stra\ss{}e des 17. Juni 136, 10623 Berlin, Germany}
\thanks{{\bf Acknowledgements:} Supported by DFG-Internationales Graduiertenkolleg “Stochastics and Real World Models”, the SFB-701 and the BiBoS-Research Center. \\
The author would like to thank Michael R\"ockner for valuable discussions and comments. Several helpful comments by the anonymous referee are gratefully acknowledged.}
\keywords{stochastic partial differential equations, stochastic porous medium equation, stochastic $p$-Laplace equation, random dynamical systems, random attractors, strictly stationary solutions, regularization by noise, ergodicity.}
\subjclass[2010]{37L55, 60H15; 76S05, 35J92, 35K57}

\begin{document}
\begin{abstract}
  We prove the existence of random attractors for a large class of degenerate stochastic partial differential equations (SPDE) perturbed by joint additive Wiener noise and real, linear multiplicative Brownian noise, assuming only the standard assumptions of the variational approach to SPDE with compact embeddings in the associated Gelfand triple. This allows spatially much rougher noise than in known results. The approach is based on a construction of strictly stationary solutions to related strongly monotone SPDE. Applications include stochastic generalized porous media equations, stochastic generalized degenerate $p$-Laplace equations and stochastic reaction diffusion equations. For perturbed, degenerate $p$-Laplace equations we prove that the deterministic, $\infty$-dimensional attractor collapses to a single random point if enough noise is added. 
  (The final publication is available at \href{http://link.springer.com/article/10.1007%2Fs10884-013-9294-5}{link.springer.com}).
\end{abstract}

\maketitle


\section{Introduction}

We study the long time behavior of solutions to SPDE of the form
\begin{equation}\label{eqn:SPDE}
  dX_t = A(t,X_t)dt + dW_t + \mu X_t \circ d\b_t,
\end{equation}
with drift $A$ satisfying a superlinear/degenerate coercivity property (cf.\ $(A3)$ below) and $\mu \in \R$. Here, $\circ$ is the Stratonovich stochastic integral. Our analysis will be based on a variational formulation of \eqref{eqn:SPDE} with respect to a Gelfand triple $V \subseteq H \subseteq V^*$. We extend known results on the existence of random attractors for quasilinear SPDE by allowing spatially much rougher noise, which in applications corresponds to assuming less spatial correlations. More precisely, we only require the Wiener process $W$ to take values in $H$ which is the natural choice of noise as far as trace-class noise is considered. This generalizes the results given in \cite{BGLR10,GLR11}. In addition, we treat joint additive and real, linear multiplicative noise, which causes difficulties in establishing asymptotic a priori bounds (i.e.\ proving bounded absorption). The key point is to realize that starting from probabilistic (variational) solutions it is possible to construct {\it $\o$-wisely} strictly stationary solutions (cf.\ Definition \ref{def:strict_stat} below).

Probabilistic methods to analyze the long-time behavior of \eqref{eqn:SPDE} in terms of the ergodicity of the corresponding Markovian semigroup often require non-degeneracy conditions for the noise. On the other hand, so far the construction of stochastic flows for quasilinear SPDE required spatial smoothness of the noise, thus excluding non-degenerate noise in many cases. Our construction remedies this obstacle and hence allows to combine ergodicity results with the essentially pathwise methods from the theory of random dynamical systems (RDS). We emphasize this point by proving that the $\infty$-dimensional attractor of perturbed, degenerate, deterministic $p$-Laplace equations collapses to a single random point if non-degenerate additive noise is included. While such a regularizing effect of additive noise in terms of reduction of the (fractal) dimension of the attractor is well-known for several systems with corresponding finite-dimensional deterministic attractor (cf.\ e.g.\ \cite{CCLR07,CS04,KS04} and Section \ref{sec:str_mix} below), the regularizing effect observed in this paper is stronger in the sense that an $\infty$-dimensional deterministic attractor is shown to reduce to a zero-dimensional random attractor.

The generation of RDS associated to SPDE is usually proven by transforming the SPDE into a random PDE. In case of SPDE driven by additive noise the transformation is based on subtracting the noise, which requires the noise to take values in the domain of the drift. However, for semilinear equations
  $$dX_t = \left(AX_t + F(X_t) \right)dt + dW_t,$$
it is well known that in the construction of an associated RDS the noise $W$ may be allowed to be an $H$-valued process by instead subtracting the stationary Ornstein-Uhlenbeck process corresponding to the linear SPDE
  $$dX_t = AX_tdt + dW_t.$$
We generalize this method to quasilinear SPDE by constructing strictly stationary solutions to nonlinear SPDE of the form
  $$dX_t = M(X_t)dt + dW_t,$$
where $M: V \to V^*$ is a strongly monotone operator. This enables us to use the regularizing property with respect to the noise that is built into the variational approach to SPDE. More precisely, for noise $W$ taking values in $H$ the strictly stationary solution takes values in the domain of $A$, i.e.\ in $\mcD(A)=V \subseteq H$ almost surely.

The standard transformation of \eqref{eqn:SPDE} into a random PDE proceeds by first canceling the multiplicative noise, by setting $X^{(1)}(t,s;\o)x := e^{-\mu \b_t(\o)}X(t,s;\o)x$ and then by subtracting the resulting additive noise $X^{(2)}(t,s;\o)x := X^{(1)}(t,s;\o)x - \int_s^t e^{-\mu\b_r} \circ dW_r(\o)$. Thus, we have used the transformation $\td T(t,s;\o)x := e^{-\mu \b_t(\o)}x - \int_s^t e^{-\mu\b_r} \circ dW_r(\o)$ and $X^{(2)}$ solves an $\o$-wise random PDE which can be used to deduce $\o$-wise asymptotic a priori bounds needed to prove the existence of a random attractor. However, due to the $s$-dependence of the transformation mapping, asymptotic bounds for $X^{(2)}$ (for $s$ small enough) do not imply analogous bounds for $X$ since $\td T^{-1}(t,s;\o)x$ is not necessarily uniformly bounded for all $s$ small enough. Instead of subtracting the integral $\int_s^t e^{-\mu \b_r} \circ dW_r$ itself, it is therefor crucial for the analysis of the long-time behavior of \eqref{eqn:SPDE} to base the transformation on a (nonlinear) strictly stationary solution $u: \R \times \O \to H$ corresponding to
    $$ u_t = u_s + \int_s^t M(u_r)dr + \int_s^t e^{-\mu z_r}\circ dW_r,\quad \P\text{-a.s. }\forall s \le t,$$
where $z_t$ is the stationary Ornstein-Uhlenbeck process for $dz_t = -z_tdt + d\b_t$. The corresponding transformation $T(t;\o)x := e^{-\mu z_t(\o)}x - u_t(\o)$ does not depend on the initial time $s$ anymore. In order to deduce asymptotic bounds we require bounds on the growth of $u_s$ as $s \to -\infty$. Such bounds may be derived from Birkhoff's ergodic theorem as soon as we have proven $u_t$ to be strictly stationary (i.e.\ $u_t(\o) = u_0(\t_t\o)$). Again, this technique is applicable only for the stationary solution $u$, not for the integral $\int_s^t e^{-\mu\b_r} \circ dW_r$ itself.

The usual approach to prove the existence of random attractors consists of two steps. First, the existence of a bounded attracting set is proven which corresponds to a uniform bound on the $H$-norm of the flow. In the second step this is used to prove the existence of a compact attracting set using bounds on some stronger norm $\|\cdot\|_S$ such that $S \hookrightarrow H$ is compact. The control of the stronger norm $\|\cdot\|_S$ requires the noise to be sufficiently smooth. In case of stochastic porous media equations (SPME) this essentially leads to the assumption of $\D W$ taking values in $V=L^{\a}(\mcO)$ \cite{GLR11} or less restrictive $\nabla W$ taking values in $V=L^\a(\mcO)$ \cite{BGLR10}, while for stochastic reaction diffusion equations (SRDE) \cite{CF94} requires $W \in H^2(\mcO) \cap H_0^1(\mcO)$. For more details we refer to Section \ref{sec:appl}. Adapting a method from \cite{CCD99} we prove compactness of the stochastic flow based only on the standard coercivity assumption of the variational approach to SPDE and on the compactness of the associated Gelfand triple $V \subseteq H \subseteq V^*$. Control of a stronger $\|\cdot\|_S$-norm thus becomes obsolete and we can allow rougher noise taking values in $H$.

In addition, we can drop the approximative coercivity property of the drift \cite[$(H5)$]{GLR11}, which allows to treat \textit{generalized} porous media equations and \textit{generalized} $p$-Laplace equations. In particular, the framework of \cite{BGLR10} is fully covered. For more details on SPME we refer to  \cite{BDPR08,BDPR08-2,BDPR09,DPR04,DPRRW06,G11c,K06,RRW07,RW08} and the references therein.

The long-time behavior of SPDE of the form \eqref{eqn:SPDE} with degenerate drift fundamentally differs from the case of singular drifts (as e.g.\ stochastic fast diffusion equations). The case of singular SPDE has been treated in \cite{G11d}. These differences can already be observed in the deterministic case in which the solutions to the porous medium equation decay polynomially (cf.\ \cite{A81}), while the fast diffusion equation exhibits finite time extinction (cf.\ \cite{V07} and references therein).
For further deterministic results on the existence of attractors for degenerate PDE we refer to \cite{CCD99,CG01,CG03,CD00,CR10,LYZ10,TY03,W10,YSZ07} and the references therein.

In Section 1 we will recall the basic notions and results concerning stochastic flows, RDS and random attractors. In Section 2 we will provide our precise framework and state the main results. Applications to various SPDE are given in Section 3 and proofs of the results are provided in Section 4.

\section{Stochastic Flows and RDS}
We recall the framework of stochastic flows, RDS and random attractors. For more details we refer to \cite{A98,CDF97,CF94,S92}. Let $(H,d)$ be a complete separable metric space  and $(\O,\F,\P,\{\t_t\}_{t \in \R})$ be a metric dynamical system, i.e.\ $(t,\o) \mapsto \theta_t(\o)$ is $(\mcB(\R) \otimes \mcF,\mcF)$-measurable, $\theta_0 =$ id, $\theta_{t+s} = \theta_t \circ \theta_s$ and $\theta_t$ is $\P$-preserving for all $s,t \in \R$.

\begin{definition}
  A family of maps $S(t,s;\o): H \to H$, $s \le t$ is said to be a stochastic flow, if for all $\o \in \O$ 
  \begin{enumerate}
    \item[(i)] $S(s,s;\o) = id_H$, for all $s \in \R$.
    \item[(ii)] $S(t,s;\o)x = S(t,r;\o)S(r,s;\o)x$, for all $t \ge r \ge s$, $x \in H$.
  \end{enumerate}
  A stochastic flow $S(t,s;\o)$ is called 
  \begin{enumerate}
    \item[(iii)] continuous if $x \mapsto S(t,s;\o)x$ is continuous for all $s \le t$, $\o \in \O$. 
    \item[(iv)] measurable if $\o \to S(t,s;\o)x$ is measurable for all $t \ge s$, $x \in H$.
    \item[(v)] a cocycle with respect to $\t_t$ if $ S(t,s;\o)x = S(t-s,0;\t_s\o)x$ for all $x \in H$, $t \ge s$, $\o \in \O$. 
  \end{enumerate} 
\end{definition}

The concept of RDS is closely related to cocycle stochastic flows. Let $\vp$ be an RDS, i.e.\ $\vp: \R_+ \times \O \times H \to H$ measurable, $\vp(0,\o) = id_H$ and 
  $$\vp(t+s,\o) = \vp(t,\t_s \o) \circ \vp(s,\o),\quad \forall s,t \ge 0,\ \o \in \O.$$
Then $S(t,s;\o) := \vp(t-s,\t_s \o)$ is a measurable cocycle stochastic flow. Conversely, let $S(t,s;\o)$ be a cocycle stochastic flow and $(t,\o,x) \mapsto S(t,0;\o)x$ be measurable. Then $\vp(t,\o) := S(t,0;\o)$ defines an RDS. 

\begin{definition}
  A function $f: \R \to \R_+$ is said to be
  \begin{enumerate}
   \item tempered if $\lim_{r \to -\infty} f_r e^{\eta r} = 0$ for all $\eta > 0$.
   \item exponentially integrable if $f \in L^1_{loc}(\R;\R_+)$ and $\int_{-\infty}^t f_r e^{\eta r}dr < \infty$ for all $t \in \R$, $\eta > 0$.
  \end{enumerate}
\end{definition}
We note that the product of two tempered functions is tempered and the product of a tempered and an exponentially integrable function is exponentially integrable if it is locally integrable. For two subsets $A,B \subseteq H$ we define
  $$ d(A,B) := \begin{cases}
                 \sup\limits_{a\in A} \inf\limits_{b \in B} d(a,b),& \text{ if } A \ne \emptyset \\
                 \infty, &                                \text{ otherwise.}
               \end{cases} $$
In the following let $S(t,s;\o)$ be a stochastic flow.

\begin{definition}
  A family $\{D(t,\o)\}_{t \in \R,\o \in \O}$ of subsets of $H$ is said to be
  \begin{enumerate}
   \item a random closed set if it is $\P$-a.s.\ closed and $\o \to d(x,D(t,\o))$ is measurable for each $x \in H$, $t \in \R$. In this case we also call $D$ measurable.
   \item right lower-semicontinuous if for each $t \in \R$, $\o \in \O$, $y \in D(t,\o)$ and $t_n \downarrow t$ there is a sequence $y_n \in D(t_n,\o)$ such that $y_n \to y$ or equivalently $d(y,D(t_n,\o))\to 0$. 
  \item tempered if $t \mapsto \|D(t,\o)\|_H$ is a tempered function for all $\o \in \O$ (assuming $H$ to be a normed space).
  \item strictly stationary if $D(t,\o) = D(0,\t_t\o)$ for all $\o \in \O$, $t \in \R$.
  \end{enumerate}
\end{definition}

From now on let $\mcD$ be a system of families $\{D(t,\o)\}_{t \in \R,\o \in \O}$ of subsets of $H$. 

\begin{definition}\label{def:abs}
  A family $\{K(t,\o)\}_{t \in \R,\o \in \O}$ of subsets of $H$ is said to be
  \begin{enumerate}
   \item $\mcD$-absorbing, if there exists an absorption time $s_0 = s_0(\o,D,t)$ such that
     \[ S(t,s;\o)D(s,\o) \subseteq K(t,\o),\quad \forall s \le s_0, \]
   \item $\mcD$-attracting, if 
     \[ d(S(t,s;\o)D(s,\o),K(t,\o)) \to 0,\quad s \to -\infty \]
  \end{enumerate}
  for all $D \in \mcD$, $t \in \R$ and $\o \in \O_0$, where $\O_0 \subseteq \O$ is a subset of full $\P$-measure.
\end{definition}

\begin{definition}\label{def:asympt_cmpt_flow}
  A stochastic flow $S(t,s;\o)$ is called
  \begin{enumerate}
   \item $\mcD$-asymptotically compact if there is a $\mcD$-attracting family $\{K(t,\o)\}_{t\in\R,\o \in \O}$ of compact subsets of $H$,
   \item compact if for all $t > s$, $\o \in \O$ and $B \subseteq H$ bounded, $S(t,s;\o)B$ is precompact in $H$.
  \end{enumerate}  
\end{definition}

We define the $\O$-limit set by
  \[ \O(D,t;\o) := \bigcap_{r < t} \overline{ \bigcup_{\tau < r} S(t,\tau;\o)D(\tau,\o)} \]
and observe
  \[ \O(D,t;\o) = \{x \in H|\ \exists s_n \to -\infty,\ x_n \in D(s_n,\o) \text{ such that } S(t,s_n;\o)x_n \to x\}. \]

\begin{definition}\label{def:ra}
  Let $S(t,s;\o)$ be a stochastic flow. A family of subsets $\{\mcA(t,\o)\}_{t \in \R,\o \in \O}$ of $H$ is called a $\mcD$-random attractor for $S(t,s;\o)$ if it satisfies $\P$-a.s.
  \begin{enumerate}
   \item $\mcA(t,\o)$ is nonempty and compact for each $t \in \R$.
   \item $\mcA$ is $\mcD$-attracting.
   \item $\mcA(t,\o)$ is invariant under $S(t,s;\o)$, i.e. for each $s \le t$
         \[ S(t,s;\o)\mcA(s,\o) = \mcA(t,\o).\]
  \end{enumerate}
\end{definition}

In the above definition we do not require the random attractor to be strictly stationary, since we are also interested in the long-time behavior of stochastic flows that do not necessarily satisfy the cocycle property. We will observe below that the proof of existence of a random attractor for such time-inhomogeneous systems is very similar to the cocycle case. On the other hand, there are significant differences between random attractors in the sense of Definition \ref{def:ra} and strictly stationary random attractors. In particular, random attractors in the  sense of Definition \ref{def:ra} are not necessarily unique even if 
  $$ \{C \subseteq H|\ C \text{ compact}\} \subseteq \mcD, $$
in contrast to strictly stationary random attractors (cf.\ \cite{C99}). Similarly, the random attractors considered here are not necessarily connected even if the state space $H$ is. 

The following theorem providing sufficient conditions for the existence of random attractors can be found in \cite{G11d}. Let $o \in H$ be an arbitrary point in $H$.
\begin{theorem}[Existence of Random Attractors]\label{thm:suff_cond_attr}
  Let $S(t,s;\o)$ be a continuous, $\mcD$-asymptotically compact stochastic flow and let $K$ be the corresponding $\mcD$-attracting family of compact subsets of $H$. Then
  $$\mcA(t,\o) := 
    \begin{cases}
      \overline{\bigcup_{D \in\mcD} \O(D,t;\o)}        &, \text{ if } \o \in \O_0 \\
       \{o\}                                           &, \text{ otherwise}
    \end{cases}$$
  defines a random $\mcD$-attractor for $S(t,s;\o)$ and $\mcA(t,\o) \subseteq K(t,\o) \cap \O(K,t;\o)$ for all $\o \in \O_0$ (where $\O_0$ is as in Definition \ref{def:abs}). 
 
  Let now $s \mapsto S(t,s;\o)x$ be right-continuous locally uniformly in $x$ and $S(t,s;\o)$ be measurable. Further assume that there is a countable family $\mcD_0 \subseteq \mcD$ consisting of right lower-semicontinuous random closed sets such that for each $D \in \mcD$, $\o \in \O$ there is a $D_0 \in \mcD_0$ satisfying $D(t,\o) \subseteq D_0(t,\o)$ for all $t \in \R$ small enough. Then $\mcA$ is a random closed set.

  If in addition $S(t,s;\o)$ is a cocycle and $\mcD_0$ consists of strictly stationary sets, then $\mcA$ is strictly stationary.
\end{theorem}

We now recall the notion of (stationary) conjugation mappings and conjugated stochastic flows (cf.\ \cite{K01,IL01}).

\begin{definition}
  Let $(H,d),(\td H,\td d)$ be two metric spaces.
  \begin{enumerate}
   \item A family of homeomorphisms $\mcT = \{T(t,\o): H \to \td H\}_{t \in \R, \o \in \O}$ such that the maps $\o \mapsto T(t,\o)x,T^{-1}(t,\o)y$ are measurable for all $t \in \R, x \in H, y \in \td H$, is called a conjugation mapping. $\mcT$ is called stationary if $T(t,\o)=T(0,\t_t\o)$ for all $t \in \R$, $\o \in \O$. In this case we set $T(\o) := T(0,\o)$.
   \item Let $Z(t,s;\o), S(t,s;\o)$ be stochastic flows.  $Z(t,s;\o)$ and $S(t,s;\o)$ are said to be (stationary) conjugated, if there is a (stationary) conjugation mapping $\mcT$ such that
     \[ S(t,s;\o) = T(t,\o) \circ Z(t,s;\o) \circ T^{-1}(s,\o). \]
  \end{enumerate} 
\end{definition}

An easy calculation shows that stationary conjugation mappings preserve the stochastic flow and cocycle property:

\begin{proposition}\label{prop:def_conj_flow}
  Let $\mcT$ be a conjugation mapping and $Z(t,s;\o)$ be a (measurable and continuous) stochastic flow. Then
  \begin{equation*}\label{eqn:def_transformed}
    S(t,s;\o) := T(t,\o) \circ Z(t,s;\o) \circ T^{-1}(s,\o)
  \end{equation*}   
  defines a conjugated (measurable and continuous) stochastic flow. If $\mcT$ is stationary and $Z(t,s;\o)$ is a cocycle then $S(t,s;\o)$ is a cocycle. 
\end{proposition}

Next, we note that the existence of a random attractor is preserved under conjugation. The proof of compactness, invariance and attraction for the conjugated attractor is immediate. Its measurability follows from the measurable selection Theorem \cite[Theorem III.9]{CV77}.

\begin{theorem}\label{thm:conj_attractor}
  Let $S(t,s;\o), Z(t,s;\o)$ be stochastic flows conjugated by a conjugation mapping $\mcT$ consisting of uniformly continuous mappings $T(t,\o): H \to H$. Assume that there is a $\tilde{\mcD}$-attractor $\tilde{\mcA}$ for $Z(t,s;\o)$ and let 
    $$\mcD := \big\{\{T(t,\o)\td D(t,\o)\}_{t \in \R,\o\in\O}|\ \td D \in \td\mcD\big\}.$$
  Then $\mcA(t,\o) := T(t,\o)\tilde{\mcA}(t,\o)$ is a random $\mcD$-attractor for $S(t,s;\o)$. If $\tilde{\mcA}$ is measurable then so is $\mcA$.
\end{theorem}

We will require the following strong notion of stationarity:
\begin{definition}\label{def:strict_stat}
  A map $X: \R \times \O \to H$ is said to satisfy (crude) strict stationarity, if 
          \[ X(t,\o) = X(0,\t_t\o),\]
  for all $\o \in \O$ and $t \in \R$ (for all $t \in \R$, $\P$-a.s., where the zero-set may depend on $t$ resp.).
\end{definition}

Since $\P$ is $\t$-invariant, crude strict stationarity implies stationarity of the law. Constructions of stationary conjugation mappings are based on strictly stationary processes. Since most constructions in stochastic analysis, like stochastic integrals and solutions to stochastic differential equations are based on limits in $L^2(\O)$ or limits in probability, usually only crude strict stationarity is satisfied. The step from stochastic analysis to RDS thus requires the selection of indistinguishable strictly stationary versions. The following Proposition is an easy adaption of \cite[Proposition 2.8]{L01}.

\begin{proposition}\label{prop:stat_perfect}
  Let $V \subseteq H$ and $X: \R \times \O \to H$ be a process satisfying crude strict stationarity. Assume that $X_\cdot \in C(\R;H) \cap L^\a_{loc}(\R;V)$ for some $\a \ge 1$, $\P$-almost surely. Then there exists a process $\tilde{X}: \R \times \O \to H$ such that
  \begin{enumerate}
   \item $\tilde{X}_\cdot(\o) \in C(\R;H) \cap L^\a_{loc}(\R;V)$ for all $\o \in \O$.
   \item $X$, $\tilde{X}$ are indistinguishable, i.e.
          \[ \P[ X_t \ne \tilde{X}_t, \text{ for some } t \in \R  ] = 0, \]
         with a $\theta$-invariant exceptional set.
   \item $\tilde{X}$ is strictly stationary.
  \end{enumerate}
\end{proposition}

\section{Setup and Main Results}\label{sec:main_result}

As mentioned in the introduction we will consider SPDE of the form 
\begin{equation}\label{eqn:spde2}
  dX_t = A(t,X_t)dt + dW_t + \mu X_t \circ d\b_t,
\end{equation}
where $\mu \in \R$, $W: \R \times \O \to H$ is a trace class two-sided Wiener process in $H$ with covariance $Q$ and $\b: \R \times \O \to \R$ is an independent two-sided real valued Brownian motion. The drift operator $A$ will be specified below. Note that $A$ may depend on a random parameter $\o \in \O$ which for simplicity we suppressed in the notation of \eqref{eqn:spde2}.

In this section we will state the main results and their assumptions, while the proofs are postponed to Section 4.

Let $H$ be a separable Hilbert space and $V$ be a reflexive Banach space continuously and densely embedded in $H$. This yields the Gelfand triple
  \[ V \subseteq H \subseteq V^*. \]
In particular, there is a constant $\l > 0$ such that $\|v\|_H^2 \le \l \|v\|_V^2$.

We assume $(W,\b)$ to be given by their canonical realization on $\O := C(\R;H\times\R)$ with the canonical filtration $\{\mcF_t\}_{t\in\R}$ and Wiener shifts $\{\t_t\}_{t \in \R}$. Let $\P$ be the law of $(W,\b)$ on $H\times\R$. Then $(\O,\mcF,\{\mcF_t\}_{t\in \R},\{\t_t\}_{t\in \R},\P)$ is an ergodic metric dynamical system. We denote the completion of $(\O,\F,\{\mcF_t\}_{t \in \R},\P)$ by $(\O,\bar\F,\{\bar\mcF_t\}_{t \in \R},\P)$.

Assume that
  \[ A: \R \times V \times \O \to V^*, \]
is such that $A(\cdot,\cdot;\o): \R \times V \to V^*$ is $(\mcB(\R) \otimes \mcB(V),\mcB(V^*))$-measurable for each $\o \in \O$. We extend $A$ by $0$ to $\R \times H \times \O$. Further, we assume that there are pathwise right-continuous mappings $C: \R \times \O \to \R$, $c: \R \times \O \to \R_+\backslash\{0\}$, a function $f: \R \times \O \to \R$ with $f(\cdot,\o) \in L^1_{loc}(\R)$ and a constant $\alpha \ge 2$ such that
\begin{enumerate}
 \item [$(A1)$] (Hemicontinuity) The map
        $$ s \mapsto { }_{V^*} \< A(t,v_1+s v_2;\o),v \>_V$$
      is continuous on $\mathbb{R}$,
 \item [$(A2)$] (Monotonicity) 
      $$2{  }_{V^*}\<A(t,v_1;\o)-A(t,v_2;\o), v_1-v_2 \>_V \le C(t,\o) \|v_1-v_2\|_H^2, $$
 \item [$(A3)$] (Coercivity) 
      $$ 2{ }_{V^*}\<A(t,v;\o), v\>_V \le C(t,\o) \|v\|_H^2 - c(t,\o) \|v\|_V^\alpha + f(t,\o),$$
 \item[$(A4)$] (Growth) 
      $$ \|A(t,v;\o)\|_{V^*}^{\frac{\a}{\a-1}} \le C(t,\o) \|v\|_V^\alpha + f(t,\o),$$
\end{enumerate}
for all $v, v_1,v_2 \in V$, $t \in \R$ and $\o \in \O$.

\begin{remark}\label{rmk:reockner_setup}
  The usual assumptions from the variational approach to SPDE (cf.\ Appendix \ref{app:var_spde}) are slightly more restrictive then $(A1)$--$(A4)$. If we require in addition that $A$ is $\bar \mcF_t$-progressively measurable, $c,C$ are non-random, $f$ is $\bar\mcF_t$-adapted and $f \in L^1_{loc}(\R; L^1(\O))$ then \cite[Theorem 4.2.4]{PR07} implies the existence of a unique, $\bar\mcF_t$-adapted solution in the sense of Definition \ref{def:prob_soln} to \eqref{eqn:spde2}. Note that for us it will be crucial to work with the non-completed filtration $\mcF_t$.
\end{remark}

Throughout this paper we will work with the convention that $C,\td C: \R \times \O \to \R$, $c,\td c: \R \times \O \to \R\setminus\{0\}$ are generic pathwise right-continuous functions and $f,\td f$ are generic functions pathwise in $L^1_{loc}(\R)$, each of which is allowed to change from line to line.

\subsection{Strictly stationary solutions}

As outlined in the introduction, the construction of  strictly stationary solutions is a key ingredient for the construction of stochastic flows for $H$-valued noise and for the analysis of their long-time behavior. In this section we will consider strictly monotone SPDE of the form
\begin{equation}\label{eqn:stricly_mon_SPDE}
  dX_t = M(t,X_t)dt + B_t dW_t,
\end{equation}
where $M: \R \times V \times \O \to V^*$ and $B:  \R \times \O \to L_2(U,H)$ satisfy $(H1)$--$(H4)$ (cf.\ Appendix \ref{app:var_spde}) with respect to $(\O,\bar\mcF,\{\bar\mcF_t\}_{t \in \R})$ and $M$ is strongly monotone, i.e.\ there exists a $c > 0$ such that
  \begin{enumerate}
    \item [$(H2')$] (Strong Monotonicity)
        $$2{  }_{V^*}\<M(t,v_1;\o)-M(t,v_2;\o), v_1-v_2 \>_V \le -c \|v_1-v_2\|_V^\a,$$
    for all $v_1,v_2 \in V, t \in \R, \o \in \O$, where $\a$ is as in $(H3)$ (Appendix \ref{app:var_spde}). 
  \end{enumerate}
Let $X(t,s;\o)x$ denote the unique solution to \eqref{eqn:stricly_mon_SPDE} starting at time $s$ in $x \in H$, given by Theorem \ref{thm:var_ex}.
  
\begin{theorem}\label{thm:nonlinear_OU}[Strictly stationary solutions]\ \\
  Let $M, B$ as above. If $\a=2$ additionally assume that $t \mapsto \E f_t$ is exponentially integrable. Then 
  \begin{enumerate}
   \item There exists an $\bar\mcF_t$-adapted, $\mcF$-measurable process $u \in L^2(\O;C(\R;H)) \cap L^\a(\O;L^\a_{loc}(\R;V))$ such that
         \begin{equation*}
           \lim_{s \to -\infty} X(t,s;\cdot)x = u_t,
         \end{equation*}  
         in $L^2(\O;H)$ for all $t \in \R$, $x \in H$. 
   \item $u$ solves \eqref{eqn:stricly_mon_SPDE} in the following sense:
         \begin{equation}\label{eqn:all_time_spde}
           u_t = u_s + \int_s^t M(r,u_r)dr + \int_{s}^t B_r dW_r,\ \P \text{-a.s., }  \forall t \ge s.
         \end{equation} 
   \item If $M$, $B$ are strictly stationary, then $u$ can be chosen to be strictly stationary with continuous paths in $H$ and satisfying $u_\cdot(\o) \in L^\a_{loc}(\R;V)$, for all $\o \in \O$. 

      If moreover, $t \mapsto \E f_t$ is exponentially integrable, then $t \mapsto \|u_t(\o)\|_V^\a$ is $\P$-a.s.\ exponentially integrable: For each $\eta > 0$, $t \in \R$ there is a $C(\eta) > 0$ such that
      \begin{equation}\label{eqn:u_V_bound}
        \E \int_{-\infty}^t e^{\eta r}\|u_r\|_V^\a dr \le C(\eta) \int_{-\infty}^t e^{\eta r} (1+\E f_r) dr. 
      \end{equation}
  \end{enumerate}
\end{theorem}

For a related result for semilinear SPDE we refer to \cite{CKS04}. In order to analyze the random differential equations obtained by stationary transformations based on the stationary solutions constructed above, we will need some growth properties for them. 

\begin{theorem}\label{thm:nonlinear_OU_prop}
  Let the assumptions of Theorem \ref{thm:nonlinear_OU} be satisfied with $M$, $B$ being strictly stationary, $k \in \N$ and $t \mapsto \E f_t^\frac{k-2+\a}{\a}$ c\`adl\`ag. If $\a = 2$ additionally assume $t \mapsto \E f_t^\frac{k}{2}$ to be exponentially integrable. Then
  \begin{enumerate}
   \item There exists a $\t$-invariant set $\O_0 \subseteq \O$ of full $\P$-measure such that for $\o \in \O_0$
          \begin{equation}\label{eqn:OU_bound}
            \frac{1}{t} \int_0^t \|u_r(\o)\|_H^k dr \to \E\|u_0\|_H^k,\quad t \to \pm\infty.
          \end{equation}
   \item $\|u_t(\o)\|_H^k$ growths sublinearly, i.e.
        \[ \lim_{t \to \pm\infty} \frac{\|u_t(\o)\|_H^k}{|t|} = 0.\]
  \end{enumerate}
\end{theorem}

\subsection{Generation of stochastic flows}

We now return to the analysis of \eqref{eqn:spde2} with $A$ satisfying $(A1)$--$(A4)$.

\begin{definition}\label{def:soln_pathw}
  A continuous, $H$-valued, $\bar\mcF_t$-adapted process $\{S(t,s;\o)x\}_{t \in [s,\infty)}$ is a solution to \eqref{eqn:spde2} if for a.a.\ $\o \in \O$: $S(\cdot,s;\o)x \in L^\a_{loc}([s,\infty);V)$ and 
    $$ S(t,s;\o)x = x + \int_s^t A(r,S(r,s;\o)x;\o) dr + \int_s^t \mu S(r,s;\cdot)x \circ d\b_r(\o) + W_t(\o)-W_s(\o),$$
  for all $t \ge s$, where $\a$ is as in $(A3)$.
\end{definition}

In order to transform \eqref{eqn:spde2} into a random PDE using strictly stationary solutions we need to require 
\begin{enumerate}
  \item [$(V)$] There is an operator $M: V \to V^*$ satisfying $(H1)$--$(H4)$, $(H2')$ with the same coercivity exponent $\a$ as for $A$ in $(A3)$.
\end{enumerate} 
This is an assumption on the underlying Gelfand triple $V \subseteq H \subseteq V^*$ rather than on the actual SPDE. If $V$ is a Hilbert space and $\a = 2$ (as e.g.\ for SRDE with sublinear reaction term) then $M$ can be chosen to be the Riesz map $i_V: V \to V^*$. For many degenerate SPDE like stochastic generalized porous media equations and stochastic generalized degenerated $p$-Laplace equations such an operator $M$ can also be easily found. For example, for the triple $V=L^{\a}(\mcO) \subseteq H=(H_0^1(\mcO))^* \subseteq V^*$ one can choose $M(v) = \D |v|^{\a-2}v$. 

The operator $M$ is used to construct stationary solutions corresponding to $dX_t = M(X_t)dt + dW_t$ that take values in $V$ while $W$ takes values in $H$. If already $W \in V$ then this regularizing property is not needed and we can just choose $M=-Id_H$. Thus condition $(V)$ can be dropped in this case.

Let $X(t,s;\o)x$ denote a variational solution to \eqref{eqn:spde2} starting in $x$ at time $s$ (which in case of Remark \ref{rmk:reockner_setup} exists and is unique). In order to associate a stochastic flow to \eqref{eqn:spde2} we transform the SPDE into a random PDE. First, we use a transformation to cancel the multiplicative noise, then the additive part will be dealt with. Let $z_t$ be the strictly stationary solution to 
 \[  dz_t  = - z_t dt + d\b_t,\]
given by $z_t(\o) = \int_{-\infty}^0 e^{s}\b_s(\t_t\o)ds$. Then $\mu_t := e^{-\mu z_t}$ satisfies
  \[ d\mu_t = \mu \mu_t z_t dt - \mu \mu_t \circ d\b_t.  \]
For $\tilde{X}(t,s;\o)x := \mu_t(\o)X(t,s;\o)x$ we obtain
\begin{align*}
  (\tilde{X}(t,s;\o)x,v)_H 
  &= \mu_t (X(t,s;\o)x,v)_H \\
  &= (\mu_s x,v)_H + \int_s^t \mu_r \ _{V^*}\<A(r,{\mu}_r^{-1} \tilde{X}(r,s;\o)x),v\>_V dr \\
    &\hskip10pt + (\int_s^t \mu_r \circ dW_r,v)_H + \mu \int_s^t (\tilde{X}(r,s;\o)x,v)_H z_rdr,
\end{align*}
$\P$-a.s.\ for all $v \in V$. The diffusion coefficients $B_t := \mu_t$ satisfy $\|B_t\|_{L_2^0}^2 = \tr(Q) |\mu_t|^2 =: f_t$. Since $\E f_t = \tr(Q) \E e^{-2\mu z_t}$ is constant, all conditions of Theorem \ref{thm:nonlinear_OU} are satisfied for $M,B$ (using assumption $(V)$). Let $u_t$ be the $\mcF$-measurable, $\bar\mcF_t$-adapted, strictly stationary solution (given by Theorem \ref{thm:nonlinear_OU}) to   
  $$ du_t = - M(u_t) dt + \mu_t \circ  dW_t.$$ 
Defining $\bar{X}(t,s;\o)x := \tilde{X}(t,s;\o)x - u_t(\o)$ we get
\begin{align*}
  (\bar{X}(t,s;\o)x,v)_H 
  &= (\mu_s x - u_s,v)_H + \int_s^t \mu_r \ _{V^*}\<A (r,\mu_r^{-1} (\bar{X}(r,s;\o)x + u_r)),v\>_V dr \\
    &\hskip10pt +  \mu \int_s^t ((\bar{X}(r,s;\o)x + u_r) z_r,v)_H dr + \int_s^t \ _{V^*}\<M(u_r),v\>_V dr,\nonumber
\end{align*}
$\P$-a.s.\ for all $v \in V$. Taken together we have used the following stationary conjugation mapping
\begin{equation}\label{eqn:stat_conj}
  T(t,\o)y := \mu_t(\o)y-u_t(\o)
\end{equation}
and the conjugated process $Z(t,s;\o)x := T(t,\o)X(t,s;\o)T^{-1}(s,\o)x$ satisfies 
\begin{equation}\label{eqn:transformed_spde_0}
  Z(t,s;\o)x = x + \int_s^t \mu_r A (r,\mu_r^{-1} (Z(r,s;\o)x + u_r)) +  \mu (Z(r,s;\o)x + u_r) z_r +  M(u_r) dr
\end{equation}
as an equation in $V^*$. Let
  $$ A_\o(r,v) := 
    \begin{cases}
        \mu_r A \left(r, \mu_r^{-1} (v + u_r) \right) +  \mu u_r z_r +  M(u_r) & \text{, if } u_r \in V \\
        \mu_r A \left(r, \mu_r^{-1} v \right) & \text{, otherwise,}
    \end{cases}$$
where for simplicity we suppressed the $\o$-dependency of $\mu_r$ and $u_r$. Recall that for all $\o \in \O$ and almost all $r \in\R$, we have $u_r(\o) \in V$. Hence from \eqref{eqn:transformed_spde_0} we obtain
\begin{align}\label{eqn:transformed_spde}
  Z(t,s;\o)x = x + \int_s^t \big( A_\o(r,Z(r,s;\o)x) + \mu z_r(\o) Z(r,s;\o)x \big)\ dr.
\end{align}
In order to define the associated stochastic flow to \eqref{eqn:spde2} we will solve \eqref{eqn:transformed_spde} for each $\o \in \O$ and then set 
\begin{align}\label{eqn:def_flow}
  S(t,s;\o)x := T(t,\o)^{-1}Z(t,s;\o)(T(s,\o)x).
\end{align}
Due to the time-inhomogeneity of the drift $A$ we cannot expect the stochastic flow to be a cocycle in general. However, if the drift is strictly stationary, i.e.
  $$A(t,v;\o) = A(0,v;\t_t\o),\quad\forall (t,v,\o) \in \R\times V\times\Omega$$
then the cocycle property will be obtained. 

\begin{theorem}\label{thm:generation}
  Assume $(A1)$--$(A4)$, $(V)$. Then
  \begin{enumerate}
   \item $Z(t,s;\o)$, $S(t,s;\o)$ defined in \eqref{eqn:transformed_spde},\eqref{eqn:def_flow} are stationary conjugated continuous stochastic flows in $H$. 
   \item The maps $t \mapsto Z(t,s;\o)x$, $S(t,s;\o)x$ are continuous, $x \mapsto Z(t,s;\o)x$, $S(t,s;\o)x$ are continuous locally uniformly in $s,t$ and $s \mapsto Z(t,s;\o)x$, $S(t,s;\o)x$ are right-continuous. 
   \item If $A$ is $(\mcB(\R)\otimes\mcB(V)\otimes\mcF,\mcB(V^*))$-measurable then $Z(t,s;\o)$ and $S(t,s;\o)$ are measurable stochastic flows. \\
         If $A$ is $\bar\mcF_t$-adapted (i.e.\ $A(t,\cdot;\cdot)$ is $(\mcB(V)\otimes\bar\mcF_t,\mcB(V^*))$-measurable for all $t \in \R$) then $S(t,s;\o)x$ is a solution of \eqref{eqn:spde2} in the sense of Definition \ref{def:soln_pathw}.
   \item If $A(t,v;\o)$ is strictly stationary then $Z(t,s;\o)$ and $S(t,s;\o)$ are cocycles.
  \end{enumerate}
\end{theorem}

\subsection{Existence of a random attractor}

In order to obtain compactness of the stochastic flows constructed above we assume
\begin{enumerate}
  \item [$(A5)$] The embedding $V \subseteq H$ is compact.
\end{enumerate}

In the following let $\mcD$ be the system of all tempered sets and $\mcD^b$ be the family of all deterministic bounded sets. We obtain

\begin{theorem}[Existence of a random attractor]\label{thm:superlinear_ra}
  Assume $(A1)$--$(A5)$ and $(V)$. Then 
  \begin{enumerate}
   \item $S(t,s;\o)$ is a compact stochastic flow.  
  \end{enumerate}
  Assume that $c, C$ are independent of time $t$ and $t \mapsto f(t,\o)$ is exponentially integrable. For $\a = 2$ additionally assume $C < \frac{c}{4\l}$ in $(A3)$. Then
  \begin{enumerate}
   \item[ii.] There is a random $\mcD$-attractor $\mcA$ for $S(t,s;\o)$.
   \item[iii.] If $A$ is $(\mcB(\R)\otimes\mcB(V)\otimes\mcF,\mcB(V^*))$-measurable and $c,C$ in $(A3)$ are $(\mcF,\mcB(\R))$-measurable, then $\mcA$ is measurable.
   \item[iv.] If $A$ is strictly stationary, then there exists a measurable, strictly stationary random $\mcD^b$-attractor for $S(t,s;\o)$.
  \end{enumerate}
\end{theorem}

 The long-time behavior of solutions to \eqref{eqn:spde2} becomes especially simple under the following strong monotonicity condition
  \begin{enumerate}
    \item [$(A2')$] (Strong Monotonicity) There exist $\l: \O \to \R_+\setminus\{0\}$, $\a >  2$ such that
        $$2{  }_{V^*}\<A(t,v_1;\o)-A(t,v_2;\o), v_1-v_2 \>_V \le -\l(\o) \|v_1-v_2\|_H^\a,$$
  \end{enumerate}
for all $v_1,v_2 \in V, t \in \R, \o \in \O$. We will prove that in this case the random attractor consists of a single random point. A similar result has been obtained in \cite{GLR11} for additive noise taking values in $V$. We generalize this result by allowing $H$-valued noise and additional real linear multiplicative noise. Let $\mcD^g$ be the system of all families of subsets  $\{D(t,\o)\}_{t \in \R, \o \in \O}$ of $H$. 

\begin{theorem}\label{thm:equil}
   Let $A$ be $(\mcB(\R)\otimes\mcB(V)\otimes\mcF,\mcB(V^*))$-measurable, $\bar\mcF_t$-adapted, satisfy $(A1)$--$(A4)$, $(A2')$ and assume $(V)$. Then there exists an $\mcF$-measurable, continuous process $\eta: \R \times \O \to H$ such that $\P$-a.s.
      \[ \lim_{s \to - \infty} S(t,s;\o)x = \eta(t,\o), \ \forall\ (t,x) \in \R\times H\]
    and $\mcA(t,\o) := \{ \eta(t,\o) \}$ defines a random $\mcD^g$-attractor for $S(t,s;\o)$. The speed of convergence is estimated by
    $$ \|S(t,s;\o)x - \eta(t,\o)\|_H^2 \le
         \left( \left(\frac{\a}{2}-1 \right)\l(\o) \int_{s}^t e^{(\a-2)\mu (\b_r(\o)-\b_t(\o))}  dr \right)^{-\frac{2}{\a-2}}.$$ 

    If $A$ is strictly stationary then $\mcA$ can be chosen to be strictly stationary.
\end{theorem}

\subsection{Random attractors for strongly mixing RDS}\label{sec:str_mix}

As outlined in the introduction, due to the possibility to consider non-degenerate trace class noise in the construction of RDS in Theorem \ref{thm:generation} we may apply probabilistic results proving ergodicity of the associated Markovian semigroup\footnote{Note that while there will be an associated Markovian semigroup defined by $P_tf(x) := \E f(\vp(t;\cdot)x)$ in the applications considered in Section \ref{sec:appl}, this is not true in general. For more details we refer to \cite{C91}.} in the applications. In this section we present a general result combining ergodicity and monotonicity of RDS to show that the minimal weak point/interval random attractor collapses to a single random point. This observation is of particular significance in the case of degenerate, perturbed $p$-Laplace equations since in the deterministic case the attractor can be shown to be infinite dimensional (cf.\ Section \ref{ssec:PLpl} below).

In the following we assume that there is a cone $H_+ \subseteq H$, i.e.\ a closed, convex set such that $\l H_+ \subseteq H_+$ for all $\l \ge 0$ and $H_+ \cap -H_+ = \{0\}$. Then $V_+ := H_+ \cap V$ is a cone in $V$. This defines a partial order $\le$ on $H$ and $V$ via $x \le y$ iff $y-x \in H_+$. For two elements $x,y \in H$ we define the interval $[x,y] = \{z \in H|\ x \le z \le y\}$. We further assume that $V_+$ is $H$-solid which means that each compact set $K \subseteq V$ is contained in an interval $[x,y] \subseteq H$, that the cone $H_+$ is normal, i.e.\ there is a $c > 0$ such that for all $0 \le x \le y$ we have $\|x\|_H \le c \|y\|_H$ and the existence of a map $_+: H \to H$ such that $x_+ \ge x,0$ for all $x \in H$.

\begin{definition}
   An RDS $\vp$ is said to be 
   \begin{enumerate}
    \item strongly mixing with ergodic measure $\mu$ if 
    $$\mcL(\vp(t;\cdot)x) \xrightarrow{TV} \mu, \quad \text{ for } t \to \infty,$$
   in total variation norm, for all $x \in H$.
    \item order-preserving (or monotone), if $\vp(t;\o)x \le \vp(t;\o)y$ for all $x \le y$, $t \ge 0$, $\o \in \O$.
   \end{enumerate}
\end{definition}

For a detailed account on order-preserving RDS we refer to \cite{C02}. We will now work with a weaker concept of a random attractor, where the requirement of $\P$-a.s.\ attraction is replaced by attraction in probability and only deterministic intervals are required to be attracted. Even this so-called weak interval random attractor will turn out to be of infinite fractal dimension for perturbed $p$-Laplace equations without additive noise (cf.\ Section \ref{ssec:PLpl} below).

\begin{definition}
  Let $\vp$ be an RDS. A random closed set $\{\mcA(\o)\}_{\o \in \O}$ is said to be a weak interval random attractor if
  \begin{enumerate}
   \item $\mcA(\o)$ is nonempty and compact $\P$-a.s.
   \item For every $x \le y$, $x,y \in H$
        \begin{equation}\label{eqn:interval_attr}
           d(\vp(t;\cdot)[x,y],\mcA(\cdot)) \xrightarrow{p} 0, \quad \text{for t} \to \infty
        \end{equation}
        in probability.
   \item $\mcA$ is invariant under $\vp$, i.e.\ $\vp(t,\o)\mcA(\o) = \mcA(\t_t\o)$, $\P$-a.s.
  \end{enumerate}
  A weak interval random attractor $\mcA$ is said to be minimal if for each weak interval random attractor $\td \mcA$ we have
    $$ \mcA \subseteq \td \mcA,\ \P-\text{a.s.}$$
\end{definition}

If we replace \eqref{eqn:interval_attr} by 
   $$ d(\vp(t;\cdot)x,\mcA(\cdot)) \xrightarrow{p} 0, \quad \text{for t} \to \infty $$
for all $x \in H$, then $\mcA$ is called a weak point random attractor. Obviously, each weak interval random attractor is a weak point random attractor.

\begin{remark}\label{rmk:interval=point}
   Let $\vp$ be an order preserving, strongly mixing RDS. Then each weak point random attractor is a weak interval random attractor and thus these two classes of random attractors coincide.
\end{remark}

The following theorem allows to use mixing properties of the Markovian dynamics to deduce the collapse of the minimal weak interval random attractor into a single random point.

\begin{theorem}\label{thm:collapse}
   Let $\vp$ be an order preserving, strongly mixing RDS with ergodic measure $\mu$ concentrated on $V$. Then there exists a unique, minimal weak point/interval random attractor $\mcA(\o) = \{v(\o)\}$ given by a single random point $v: \O \to H$ with $v \in V$, $\P$-almost surely. 
\end{theorem}

The proof of Theorem \ref{thm:collapse} will rely on a modification of \cite[Proposition 2]{CS04}. In the literature, several results related to Theorem \ref{thm:collapse} are known. In case of strongly monotone SPDE the dissipativity approach to existence and uniqueness of invariant measures (cf.\ e.g.\ \cite[Theorem 4.3.9]{PR07}) may be used to prove the existence of a random attractor consisting of a single random point (cf.\ Theorem \ref{thm:equil}). A similar approach has been taken in \cite{M99} to prove the existence of a stable stationary solution to the stochastic 2D-Navier-Stokes equation with additive noise and large viscosity. Since the dissipativity approach is based on strong monotonicity of the drift, the deterministic counterpart of the SPDE has a unique stable invariant solution and thus no regularization due to noise is observed.

In \cite{CCLR07} it is shown that the random attractor corresponding to the one-dimensional Chafee-Infante equation perturbed by non-degenerate additive noise
\begin{align*}
   dX_t &= (\D X_t + X_t - X_t^3) dt + dW_t,\quad \text{on } \mcO = [0,L], \text{ some } L > 0,\\
   X_t &\equiv 0,\quad \text{on } \partial\mcO,
\end{align*}
consists of a single random point. The proof relies on a known, rather strong bound on the random attractor $\mcA$, namely its order-boundedness: There are random variables $\underline{a},\overline{a}$ in $L^2(\mcO)$ such that
  $$\underline{a}(\o) \le \mcA(\o) \le \overline{a}(\o),\quad \forall \o \in \O.$$
Such a property is not known for the SPDE considered in this paper. On the other hand, the method used in \cite{CCLR07} only requires uniqueness of the invariant Markov measure, while we need to suppose that the RDS is strongly mixing. 

For strongly mixing, white noise RDS $\vp$ satisfying an asymptotic, compact absorption property it is shown in \cite{KS04} that there is a minimal random point attractor given by
  $$ \mcA(\o) = \supp \mu_\o,\quad \P\text{-a.s.},$$
where $\mu_\o := \lim_{t\to\infty}\vp(t,\t_{-t}\o)\mu$ exists $\P$-a.s.\ in the weak$^*$-topology due to \cite{C91}. If $\vp$ is order preserving and $\mu(V) = 1$ then we obtain by Theorem \ref{thm:collapse} that $\supp \mu_\o$ is single-valued, in other words $\mu_\o = \d_{v(\o)}$. 

For related results for SDE we refer to \cite{C00,CF98} and the references therein.

\section{Examples}\label{sec:appl}

\subsection{Stochastic Generalized \texorpdfstring{$p$}{p}-Laplace Equation}\label{ssec:PLpl}

In consistency with our general results, $\a$ will take the role of $p$. Let $(M,g,\nu)$ be a $d$-dimensional weighted compact smooth Riemannian manifold with boundary equipped with Riemannian metric $g$, associated volume measure $\mu$ and $d\nu(x) := \s(x) d\mu(x)$ with $\s$ being a smooth, positive function on $M$. Further, let $\a > 2$ and $V := W_0^{1,\a}(M,\nu) \subseteq H := L^2(M,\nu)$, which is a compact embedding. We denote the inner product on $T_xM$ given by the Riemannian metric $g$ by $(\cdot,\cdot)_x$ and the associated norm by $|\cdot|_x$. Let $W_t$ be an $H=L^2(M,\nu)$-valued Wiener process. 
We consider the degenerate $p$-Laplace equation
  \begin{equation}\label{eqn:SpLP}
     dX_t = \left( \div_\nu(\Phi(\nabla X_t)) + G(X_t) + g(t) \right)dt + dW_t + \mu X_t \circ d\b_t,
  \end{equation}
  where $\Phi: M \times TM \times \O \to TM$ is measurable, $\Phi(x,\cdot,\o):T_xM \to T_xM$ is continuous and
  \begin{equation*}\begin{split}
     (\Phi(x,\xi,\o)-\Phi(x,\td\xi,\o),\xi-\td \xi)_x &\ge \l(\o)|\xi-\td\xi|_x^\a \\
     (\Phi(x,\xi,\o),\xi)_x &\ge c(\o)|\xi|_x^\a - f(\o)\\
     |\Phi(x,\xi,\o)|_x^\frac{\a}{\a-1} &\le C(\o) |\xi|_x^\a + f(\o), \quad \forall x \in M,\ \xi,\td\xi \in T_xM,\ \o \in \O,
  \end{split}\end{equation*}
  with $f: \O \to \R$ being measurable, $c$ measurable, positive and $C,\l$ measurable, non-negative. In particular, this includes the standard nonlinearity $\Phi(x,\xi) := |\xi|_x^{\a-2}\xi$. Further, let $G:\R\times\O\to\R$ be measurable, pathwise continuous with
  \begin{equation*}\begin{split}
      (G(u,\o)-G(v,\o))(u-v)    &\le C(\o)|u-v|^2 \\
      |G(u,\o)|^\frac{q}{q-1}   &\le C(\o)(1+|u|^q),\quad\forall v,u \in \R,\ \o \in \O,
  \end{split}\end{equation*}
  for some $q \in (1,\a)$ and $g: \R \times \O \to H$ be measurable, exponentially integrable in $t$. 

  The $p$-Laplace operator then maps $V\times\O \to V^*$ by 
    $$A(v,\o)(w) = -\int_M (\Phi(x,\nabla v,\o),\nabla w)_x d\nu(x),\quad v,w \in V,\ \o \in \O.$$
  We obtain

 \begin{example}[Generalized Stochastic $p$-Laplace Equation]\label{exam:SpLE}
    There is an associated compact stochastic flow $S(t,s;\o)$ to \eqref{eqn:SpLP} with a measurable, random $\mcD$-attractor $\mcA$. If $g \equiv 0$ then $S(t,s;\o)$ is a cocycle and there is a measurable, strictly stationary random $\mcD^b$-attractor $\mcA^b$. If $\l(\o) > 0$ and $G \equiv 0$ then $\mcA$ consists of a single random point.
\end{example}

\begin{remark}\label{rmk:SPLE}
  In comparison to our results, in \cite{GLR11} the existence of a random attractor for 
    $$dX_t= \left( \div\left(|\nabla X_t|^{\a-2}\nabla X_t\right) + G(X_t) \right)dt + dW_t$$
 with Neumann boundary conditions on convex, open, bounded domains $\mcO \subseteq \R^d$ and noise taking values in $W^{3,2}(\mcO)$ has been proven.

  In \cite{ZL11} the existence of a random attractor for similar equations of the form
    \[ dX_t = \left(-\D\left(|\D X_t|^{\a-2}\D X_t \right) + |X_t|^q X_t \right)dt + W_t, \]
  with Dirichlet boundary conditions for $X_t$ and $\nabla X_t$, $2 \leq q \leq \a$ has been obtained for Wiener noise $W$ in $W_0^{4,\a}(\mcO)$.

  Our results yield the existence of a random attractor for Wiener noise in $L^2(\mcO)$ and for more general nonlinearities (with assumptions analogous to those in Example \ref{exam:SpLE}).
\end{remark}

In the following we will consider the standard $p$-Laplace equation perturbed by a linear reaction term and additive or real linear multiplicative noise with Dirichlet boundary conditions, i.e.
\begin{equation}\label{eqn:SpLP-2}
  dX_t = \div \left( |\nabla X_t|^{\a-2}\nabla X_t + \eta X_t \right)dt  + \mu X_t \circ d\b_t + \s dW_t,
\end{equation}
on a bounded, smooth domain $\mcO \subseteq \R^d$, with $\eta > 0$, $\mu \ge 0$, $\s \ge 0$ and $\a > 2$. Let $\vp$ be the RDS corresponding to \eqref{eqn:SpLP-2}, given by $\vp(t;\o)x = S(t,0;\o)x$.

\subsubsection{Infinite dimensional random attractors}

In the case $\s = 0$, i.e.\ if no additive noise is present, we will now prove that each weak interval random attractor is infinite dimensional. 

A precompact set $\mcA \subseteq H$ can be covered by a finite number of balls of radius $\ve$ for each $\ve > 0$. Let $N_\ve(\mcA)$ be the minimal number of such balls. Then, the Kolmogorov $\ve$-entropy of $\mcA$ is defined by
  $$\mathbbm{H}_\ve(\mcA) := \ln(N_\ve(\mcA)).$$
The fractal dimension of $\mcA$ is defined by
  $$ d_f(\mcA) = \limsup_{\ve \to 0} \frac{\mathbbm{H}_\ve(\mcA)}{ln(\frac{1}{\ve})}. $$
We obtain

\begin{theorem}\label{thm:infinite_dim_attr}
  Let $\mcA$ be a weak interval random attractor for $\vp$. Then, the Kolmogorov $\ve$-entropy of $\mcA$ is bounded below by 
    $$ \mathbbm{H}_\d (\mcA(\o)) \ge C(\o) \d^{\frac{-d(\a-2)}{d(\a-2)+\a}},\ \P-a.s.,$$
  where $C(\o) > 0$ is a constant which may depend on $\a, d$. In particular, the fractal dimension $d_f(\mcA(\o))$ is infinite $\P$-almost surely.
\end{theorem}
\begin{proof}
  The proof is inspired by \cite[Theorem 4.1]{EZ08}. In order to prove the lower bound on the Kolmogorov $\ve$-entropy we consider the unstable manifold of the equilibrium point $0$ defined by
  \begin{align*}
    \mcM^+(0,\o) := \{& u_0 \in H\ |\ \exists \text{ order bounded function } u: (-\infty,0] \to H, \text{ such that}\\ &\vp(t;\t_{-t}\o)u(-t) = u_0
                                    \text{ for all } t \ge 0 \text{ and } \|u(t)\|_H \to 0 \text{ for } t \to -\infty \}.
  \end{align*}
  Since $\mcA(\o)$ attracts all deterministic, order bounded sets we have
    $$ \mcM^+(0,\o) \subseteq \mcA(\o),\quad \P-\text{a.s.}$$
  Therefore, it is sufficient to derive a lower bound on the Kolmogorov $\ve$-entropy for the unstable manifold of $0$. The construction of appropriate elements in $\mcM^+(0,\o)$ will be based on a suitable time scaling of Barenblatt solutions with sufficiently small initial supports.

  First, we transform \eqref{eqn:SpLP-2} into a random PDE. Let 
    $$\bar Z(t,s;\o)x := e^{-\mu\b_t(\o)-\eta t}S(t,s;\o)\left( e^{\mu\b_s(\o)+\eta s}x \right),$$
  where $S(t,s;\o)$ is the stochastic flow associated to \eqref{eqn:SpLP-2} as constructed in Theorem \ref{thm:generation}. Then $\bar Z$ solves
  \begin{equation}\begin{split}\label{eqn:pLp_alt_transf}
      \frac{d}{dt} \bar Z(t,s;\o)x &= e^{(\a-2)(\mu\b_t(\o)+\eta t)} \div\left(|\nabla\bar Z(t,s;\o)x|^{\a-2} \nabla\bar Z(t,s;\o)x\right),\\
      \bar Z(s,s;\o)x &= x,
  \end{split}\end{equation}
  for a.e.\ $t \ge s$. In order to construct an element $u_0 \in \mcM^+(0,\o)$ we need to find an order bounded function $u: (-\infty,0] \to H$ converging to $0$ such that 
    $$u_0 = \vp(t;\t_{-t}\o)u(-t) = S(0,-t;\o)u(-t) = \bar Z(0,-t;\o)\left(e^{-\mu\b_{-t}(\o)+\eta t} u(-t) \right),$$
  for all $t \ge 0$. Since $e^{\mu\b_{-t}(\o)-\eta t} \to 0$ for $t \to \infty$ it is enough to find an order bounded (thus bounded) function $v: (-\infty,0] \to H$ such that $u_0 = \bar Z(0,-t;\o)v(-t).$

  In the following we restrict onto a subset of full $\P$-measure on which the law of iterated logarithm holds, i.e. $\o \in \O$ for which
     $$ \lim_{t \to \pm \infty} \frac{\b_t(\o)}{\sqrt{t \log \log t}} = \sqrt{2}. $$
  In order to find such an order bounded function $v$, we use a time scaling to transform \eqref{eqn:pLp_alt_transf} into the standard $p$-Laplace equation on a finite time interval. Let 
   $$f(t,\o) := e^{(\a-2)(\mu \b_t(\o)+\eta t)}, \quad F(t,\o) := -\int_t^0  f(r,\o) dr \in C^1(\R_-)$$
  and $s_0(\o) := -\lim_{t \to -\infty} F(t,\o) $. Since $F$ is strictly increasing we may define
    $$ g(t,\o) := F^{-1}(t,\o),\quad t \in (s_0(\o),0]. $$
  For ease of notation we will suppress the $\o$-dependency of the time scaling $g$ in the following. We note $g \in C^1(s_0,0)$ with $g'(t) = \frac{1}{f(g(t))} > 0$, $g(0) = 0$ and $g(t) \to -\infty$ for $t \to s_0$. Let $U(t,s;\o)x := \bar Z(g(t),g(s);\o)x$. Then
  \begin{equation}\begin{split}\label{eqn:det_pLp}
    \frac{d}{dt} U(t,s;\o)x 
    &= g'(t) f(g(t)) \div\left(|\nabla U(t,s;\o)x|^{\a-2} \nabla U(t,s;\o)x\right) \\
    &= \div\left(|\nabla U(t,s;\o)x|^{\a-2} \nabla U(t,s;\o)x\right),
  \end{split}\end{equation}
  for $\text{ a.e. } t \ge s,\ t,s \in (s_0,0]$. It is well known that for each $x \in L^\infty(\mcO)$ there is a unique solution $S(t) x$ to \eqref{eqn:det_pLp} with $S(s_0)x = x$ and $\sup_{t \in [s_0,0]} \|S(t) x\|_\infty < \infty.$ In particular, $S(t)x$ is order bounded in $H$. Uniqueness of \eqref{eqn:det_pLp} then implies 
    $$S(t)x = S(t-s)S(s)x = U(t,s;\o)S(s)x,\quad \forall s \in (s_0,t].$$
  Hence, $v(t) := S(F(t))x: (-\infty,0] \to H$ is order bounded in $H$,
  \begin{align*}
    \bar Z(0,-t;\o)v(-t) 
    &= U(0,F(-t);\o)S(F(-t))x \\
    &= S(0)x = v(0)
  \end{align*} 
  and consequently $v(0) \in \mcM^+(0,\o)$. 

  In order to use this construction of elements $v(0) \in \mcM^+(0,\o)$ to derive a lower bound on the Kolmogorov $\ve$-entropy of $\mcM^+(0,\o)$ we consider special solutions to \eqref{eqn:det_pLp} so that the final values $v(0)=S(0)x$ are sufficiently far apart (w.r.t.\ the $H$-norm).  Recall that the Barenblatt solutions, given by
    $$ U(t,\xi) := t^{-k} \left( C(M)-q|\xi|^\frac{\a}{\a-1}t^{\frac{-k\a}{d(\a-1)}}\right)_+^\frac{\a-1}{\a-2},$$
  with $k=\frac{1}{\a-2+\frac{\a}{d}}$, $q = \frac{\a-2}{\a} \left(\frac{k}{d}\right)^\frac{1}{\a-1}$ solve \eqref{eqn:det_pLp} with  
    $$ \supp U(t,\cdot) \subseteq B\left(t^\frac{k}{d}\left( \frac{C(M)}{q} \right)^\frac{\a-1}{\a},0\right).$$
  Here $C(M) = C(\a,d)M^\frac{\a(\a-2)k}{d(\a-1)}$ is a constant scaling with the total mass $M=\|U(t)\|_{L^1(\mcO)}$. For $\ve > 0$ small enough we can find a finite set $R_\ve = \{\xi_i\} \subseteq \mcO$ such that
  \begin{align*}
     B(\ve,\xi_i) \cap B(\ve,\xi_j)     &= \emptyset, \quad \text{for } i \ne j \\
     |R_\ve|                            &\ge C \ve^{-d}\\
     B(\ve,\xi_i)                       &\subset \mcO, \quad \forall i.
  \end{align*}
  Choosing $M$ small enough (i.e.\ $M = C(\a,d) \left(\ve s_0^{-\frac{k}{d}} \right)^\frac{d}{k(\a-2)}$) we can thus construct solutions to \eqref{eqn:det_pLp} based on the Barenblatt solutions by
    $$ U^{m}(t,\xi) = \sum_{i=1}^{|R_\ve|} m_i U(t-s_0(\o),\xi-\xi_i), $$
  for each $m \in \{0,1\}^{|R_\ve|}$. By the choice of $C(M)$, $U^{m}$ satisfies Dirichlet boundary conditions. For $m_1 \ne m_2$ we observe
    $$ \|U^{m_1}(0,\cdot) - U^{m_2}(0,\cdot)\|_H \ge M =  C(\a,d) \left(\ve s_0^{-\frac{k}{d}} \right)^\frac{d}{k(\a-2)}.$$
  Hence,
    $$ \mathbbm{H}_\d (\mcA(\o)) \ge \mathbbm{H}_\d (\mcM^+(0,\o)) \ge \log_2 2^{|R_\d|} \ge C(\o) \d^{-(\a-2)k}, \P-a.s.$$
  and   
    $$ d_f(\mcA(\o)) \ge d_f(\mcM^+(0,\o)) = \limsup_{\d \to 0} \frac{\mathbbm{H}_\d (\mcM^+(0,\o))}{log_2(\frac{1}{\d})} = \infty, \P-a.s.$$
\end{proof}

\subsubsection{Regularization by additive noise}

We now consider \eqref{eqn:SpLP-2} with purely additive noise ($\mu = 0$), i.e.
  $$ dX_t = \div \left( |\nabla X_t|^{\a-2}\nabla X_t + \eta X_t \right)dt + \s dW_t. $$
By Theorem \ref{thm:infinite_dim_attr} we know that for $\s = 0$ the weak (random) interval attractor is infinite dimensional. We will now use Theorem \ref{thm:collapse} to prove that the random attractor is zero dimensional if the additive noise is non-degenerate. Let $Q^\frac{1}{2} \in L(H)$ be injective and define
\begin{equation*} 
  \|u\|_{Q^\frac{1}{2}} :=
    \begin{cases}
       \|y\|_U, & Q^\frac{1}{2} y = u \\
       \infty, & \text{otherwise.}
    \end{cases} 
\end{equation*}
We obtain

\begin{corollary}\label{cor:collapse}
   Assume $d \le 2$ and 
    $$ \|u\|_V^\a \ge c \|u\|_{Q^\frac{1}{2}}^\s\|u\|_H^{\a-\s}, \quad \forall u \in V, $$
   for some $\s \ge 2$, $\s > \a-2$. Then there is a unique minimal weak interval attractor $\mcA$ consisting of a single random point, i.e.
    $$ \mcA(\o) = \{v(\o)\}, $$
   for some strictly stationary $v: \O \to H$ with $v \in V$, $\P$-almost surely.
\end{corollary}
\begin{proof}
  Let $H_+ \subseteq H$ be the cone of all $x \in H$ such that $x(\xi) \ge 0$ for almost all $\xi \in \mcO$. Then $H_+$ is normal and we can define the positive part of $x \in H$ by $x_+(\xi) := \max(x(\xi),0)$. Note that $V=W^{1,\a}_0(\mcO) \hookrightarrow C^0(\mcO)$ since $d \le 2$ and $\a \ge 2$. Therefore, bounded sets in $V$ are uniformly bounded and thus are contained in some interval in $H$. Consequently, $V_+$ is $H$-solid.

  By \cite[Theorem 1.4]{L09} the associated Markovian semigroup $P_tf(x) := \E f(\vp(t;\cdot)x)$ is strongly mixing with respect to some invariant measure $\mu$, i.e.
    $$ \mcL(\vp(t;\cdot)x) =  P_t^*\d_x \xrightarrow{TV} \mu,$$
  with $\mu(V) = 1$.

  The proof that $\vp$ is order-preserving proceeds similar to \cite{G11d}. More precisely, we consider a non-degenerate approximation $\Phi^\ve(\xi) = \xi |\xi|^{\a-2} + \ve \xi$ of $\Phi(\xi) = \xi |\xi|^{\a-2}$ and a smooth approximation of the strictly stationary solution $u_t$ given by $u_t^n := \mcP_n u_t$, where $\mcP_n$ are the nonlinear Galerkin approximants as in \cite{G11}. Let $Z^{n,\ve}x$ denote the corresponding approximating solution with initial condition $x \in C^2(\mcO)$:
    $$ dZ^{n,\ve}x = \div\left(\Phi^\ve(Z^{n,\ve}x+u^n)\right)dt.$$
  Unique existence of $Z^{n,\ve}x$ follows from standard results and comparison results for second order parabolic quasilinear equations (cf.\ e.g.\ \cite{LSU67}) yield
    $$Z^{n,\ve}x \le Z^{n,\ve}y,$$
  for all $x,y \in C^2(\mcO)$ with $x(\xi) \le y(\xi)$, for all $\xi \in \mcO$. Since the cone $H_+$ is closed and by the choice of approximation we conclude that $\vp$ is order-preserving.

  The claim now follows from Theorem \ref{thm:collapse}.
\end{proof}

\subsection{Stochastic Reaction Diffusion Equations}

Let $(E,\mcB,m)$ be a finite measure space with countably generated $\s$-algebra $\mcB$ and let $(L,\mcD(L))$ be a strictly negative-definite (i.e.\ $(Lv,v)_{L^2(m)} \le -c\|v\|^2_{L^2(m)},\ \forall v \in \mcD(L),\ \text{some } c > 0$), self-adjoint operator on $H := L^2(m)$. We set $m(fg) := \int_E fg\ dm$, for $fg \in L^1(m)$, $(f,g) := (f,g)_{L^2(m)}$ and $\|f\| := \|f\|_{L^2(m)}$. The corresponding Dirichlet space $V = \mcD(\mcE)$ is a Hilbert space with inner product $(v,w)_V := \mcE(v,w) := m(\sqrt{-L}v\sqrt{-L}w)$, $v,w \in V$. We consider the  triple
    $$ V:= \mcD(\mcE) \subseteq H := L^2(m) \subseteq V^*.$$
and assume that  $V \subseteq H$ is compact. 

Let $G: \R \times \R \times \O \to \R$ be measurable, $G(t,\cdot,\o)$ continuous for all $(t,\o) \in \R \times \O$ and
\begin{equation}\label{eqn:reaction_term}\begin{split}
  \left( G(t,u,\o)-G(t,v,\o) \right) (u-v)      &\le \l(\o) |u-v|^2 \\
  |G(t,u,\o)|^2                                 &\le  C(\o) |u|^2+f(t,\o),
\end{split}\end{equation}
with coefficient $\sqrt{C(\o)} < c$ and $f(\cdot,\o)$ being exponentially integrable. With $M(v) := Lv$ we obtain

\begin{example}\label{ex:srde}
 The stochastic flow associated to the stochastic reaction-diffusion equation
 \begin{equation}\label{rd}
    d X_t=(L X_t + G(t,X_t)) dt + dW_t + \mu X_t \circ d\b_t,
 \end{equation}
where $W$ is an $H=L^2(m)$-valued Wiener process, has a measurable random $\mcD^2$-attractor $\mcA$. If $G$ is strictly stationary, then there is a measurable, strictly stationary random $\mcD^b$-attractor $\mcA^b$. If $\l(\o) \le 0$, then $\mcA$ consists of a single random point. 
\end{example}

\begin{remark}\label{rmk:RDE}
  In Example \ref{ex:srde} we had to restrict to reaction terms $G$ of at most linear growth due to the restrictions of the variational approach to (S)PDE (cf.\ Appendix \ref{app:var_spde}). However, the main ideas also apply to SRDE with high-order reaction terms, i.e.\ let $\mcO \subseteq \R^d$ be an open, bounded domain and consider
  \begin{equation}\label{rd2}
    dX_t=(\D X_t - |X_t|^{p-2}X_t + \eta X_t) dt + dW_t + \mu X_t \circ d\b_t,\ (p > 2).
  \end{equation}     
  In the construction of strictly stationary solutions the variational methods have to be replaced by a mild approach to SPDE (cf.\ e.g.\ \cite{DPZ92}) where we choose
    $$M(v) := \D v - |v|^{p-2}v $$
  corresponding to the triple
    $$ V:=W^{1,2}_0(\mcO) \cap L^p(\mcO) \subseteq H:= L^2(\mcO)\subseteq V^*.$$
  Using the strong monotonicity of the operator $M$, strictly stationary solutions can be constructed similar to Theorem \ref{thm:nonlinear_OU}.
  The technique proving compactness of the stochastic flow by compactness of $V \subseteq H$ remains the same, which proves the existence of a random attractor for SRDE of the form \eqref{rd2} where $W$ is a Wiener process taking values in $L^2(\mcO)$. In comparison, in \cite[Section 5]{CF94} the existence of a random attractor for stochastic reaction-diffusion equations perturbed by finite dimensional Wiener noise is obtained under the assumption that the noise takes values in $H^2(\mcO) \cap H^1_0(\mcO)$.
\end{remark}

\subsection{Stochastic Porous Media Equation}

Let $(E,\mcB,m)$ be as above and let $(L,\mcD(L))$ be a negative-definite, self-adjoint operator on $L^2(m)$ with Ker$(L)=\{0\}$. Define $\mcD(\mcE) := \mcD(\sqrt{-L})$, $\mcE(u,v) := (\sqrt{-L}u,\sqrt{-L}v), \text{ for } u,v \in \mcD(\mcE)$ and $\mcF_e$ to be the abstract completion of $\mcD(\mcE)$ with respect to $\|\cdot\|_{\mcF_e}^2 := \mcE(\cdot,\cdot)$.

Let $\Phi: \R \times \R \times \O \to \R$ be measurable with $\Phi(t,0,\o) = 0$, $\Phi(t,\cdot,\o) \in C(\R)$ and
\begin{align*}
  (\Phi(t,r,\o)-\Phi(t,s,\o))(r-s)      &\ge \l(\o)|r-s|^\a\\
  \Phi(t,r,\o)r                         &\ge c(\o)|r|^\a - f(t,\o) \\
  |\Phi(t,r,\o)|^\frac{\a}{\a-1}        &\le C(\o)|r|^\a + f(t,\o),\ \forall s,r,t \in \R, \o \in \O,
\end{align*}
for some $\a > 2$, $c$ positive, $\l,C$ non-negative and $f(\cdot,\o)$ exponentially integrable. In particular, the standard nonlinearity $\Phi(r) := |r|^{\a-2}r$ is included in our general framework. We assume
\begin{enumerate}
 \item[(L)] The embedding $\mcF_e \subseteq L^\frac{\a}{\a-1}(m)$ is compact and dense.
\end{enumerate}
This yields the compact Gelfand triple
  $$ V := L^{\a}(m) \subseteq  H := \mcF_e^* \subseteq V^*.$$
We now present some more concrete examples for which $(L)$ is satisfied.
\begin{example}
  Let
  \begin{enumerate}
    \item $E$ be a smooth, compact Riemannian $d$-dimensional manifold with boundary and $L$ be the Friedrichs extension of a symmetric, uniformly elliptic operator of second order on $L^2(m)$ with Dirichlet boundary conditions. For example, let $L$ be the Dirichlet Laplacian on $E$.
    \item $E \subseteq \R^d$ be an open, bounded domain, $L:=(-\D)^\b$ with its standard domain and $\b \in \frac{d}{2}\left(\frac{2-\a}{\a},1\right) \cap (0,1]$.
  \end{enumerate}
  Then $(L)$ is satisfied.
\end{example}

Choosing $M(v) := L\left(|v|^{\a-2}v\right): V \to V^*$ we obtain
\begin{example}[Stochastic Porous Media Equation] The stochastic flow associated to the stochastic porous media equation
  \begin{equation}\label{PME}
    dX_t= \left(L \Phi(t,X_t) + \eta X_t \right)dt + dW_t + \mu X_t \circ d\b_t,
  \end{equation}
  with $\eta \in \R$ and $W$ being an $H=\mcF_e^*$-valued Wiener process, has a measurable random $\mcD$-attractor $\mcA$.  If $\Phi$ is strictly stationary, then $S(t,s;\o)$ is a cocycle and there is a measurable, strictly stationary random $\mcD^b$-attractor $\mcA^b$. If $\l(\o) > 0 \ge \eta$, then the random attractor consists of a single random point.
\end{example}

\begin{remark}\label{rmk:SPDE}
  In \cite{BGLR10} the existence of random attractors for 
    $$ dX_t = \D\Phi(X_t)dt + dW_t $$
  has been proven for additive Wiener noise in $W_0^{1,\a}(\mcO)$ on open, bounded domains $\mcO \subseteq \R^d$. Here, we only require $W^{-1,2}_0(\mcO)$ regularity of the noise.
\end{remark}

\begin{remark}
  Similar arguments as for the stochastic $p$-Laplace equation show that every weak interval attractor has infinite fractal dimension in the deterministic case with $\eta > 0$ (cf.\ \cite{EZ08}). On the other hand, in the case of sufficiently non-degenerate additive noise the weak interval random attractor collapses to a single random point, which can be shown as in Corollary \ref{cor:collapse}.
\end{remark}

\section{Proofs}


\subsection{Comparison Lemma and a priori bounds}

In \cite{G11d} comparison Lemmata and a priori bounds for singular ODE have been shown. We will now prove similar results for superlinear/degenerate ODE.

\begin{lemma}[Comparison Lemma]\label{lemma:comp_1}
  Let $\b > 1$, $s \le t$, $q \ge 0$, $h \in L^1([s,t])$ nonnegative and $v: [s,t] \to \R_+$ be an absolutely continuous subsolution of
  \begin{equation}\label{eqn:comp_1_1}\begin{split}
    y'(r) &= -h(r)y(r)^\b,\ r \in [s,t] \\
    y(s)  &= q,
  \end{split}\end{equation}
  i.e.\ $v'(r) \le - h(r)v(r)^\b$ for a.e.\ $r \in [s,t]$ and $v(s) \le q$. Then
    \[ v(r) \le \left( q^{-(\b-1)} + (\b-1) \int_s^r h(\tau)d\tau  \right)^{\frac{-1}{\b-1}}, \quad \forall r \in [s,t].\]
\end{lemma}
\begin{proof}
  The proof proceeds as in \cite{G11d}, using that the unique nonnegative solution $y^\ve$ to \eqref{eqn:comp_1_1} with initial value $q+\ve$ is given by 
    $$y^\ve := \left( (q+\ve)^{-(\b-1)} + (\b-1) \int_s^r h(\tau)d\tau \right)^{\frac{-1}{\b-1}} \le q+\ve.$$
\end{proof}

\begin{lemma}[A priori bound]\label{lemma:comp_superlinear}
  Let $\b > 1$, $p: \R \to \R_+$ c\`adl\`ag, $h \in C(\R)$ positive, $q: \R \to \R_+$ and for each $s \in \R$ let $v(\cdot,s): [s,\infty) \to \R_+$ be an absolutely continuous subsolution of 
  \begin{equation}\label{eqn:comp_2_1}\begin{split}
    y'(r,s) &= -h(r)y(r,s)^\b + p(r),\ r \ge s \\
  \end{split}\end{equation}
  with $y(s,s) = q(s)$. Assume $\int_s^t h(\tau)d\tau \to \infty$, for all $t \in \R$ and $s \to -\infty$. Then, for each $t \in \R$ there is an $s_0=s_0(t,p,h) \in \R$ and an $R=R(t,p,h) > 0$ such that for all $s \le s_0$
    \[ v(t,s) \le R.\]
\end{lemma}
\begin{proof} 
  We proceed similarly to \cite{G11d}. Without loss of generality we assume $p(r) \ge \d > 0$ for some $\d > 0$ (otherwise redefine $p(r):= p(r) \vee \d$).

 Let $t\in \R$, $A(s) := \{ r \in [s,t] \ |\ \frac{h(r)}{2} v(r,s)^\b \le p(r) \}$ and $a(s) = \sup A(s) \vee s$. We first show that there exists an $s_0=s_0(t,p,h) \le t$ such that $A(s) \ne \emptyset$ for all $s \le s_0$. Let $s \le t$ such that $A(s) = \emptyset$, i.e. $ \frac{h(r)}{2} v(r,s)^\b > p(r)$, for all $r \in [s,t]$. Hence,
  \begin{align*}
    v'(r,s) 
    &\le - h(r) v(r,s)^\b + p(r) 
    \le - \frac{1}{2} h(r) v(r,s)^\b,\ \text{ for a.e. } r \in [s,t].
  \end{align*}
  By Lemma \ref{lemma:comp_1} and using the assumption $\int_s^t h(\tau)d\tau \to \infty$ we obtain
  \begin{align}
    v(t,s) 
    &\le \left( \frac{\b-1}{2} \int_s^t h(\tau)d\tau \right)^{\frac{-1}{\b-1}} \to 0, \quad \text{for } s \to -\infty.
  \end{align}  
  Since also $0 < \big( \frac{2p(t)}{h(t)}  \big)^{\frac{1}{\b}} < v(t,s)$, we conclude $s \ge s_0$ for some $s_0 = s_0(t,p,h)$. 

  Next we prove that there exists an $a_1=a_1(t,p,h) \le t$ such that $a_1\le a(s) $ for all $s \le s_0$. Let $s \le s_0$, thus $A(s) \ne \emptyset$. If $a(s) = t$ then nothing has to be shown, thus suppose $a(s) < t$. By definition of $a(s)$ and right-continuity of $v,p$ we have 
   \[ p(r) \le \frac{h(r)}{2} v(r,s)^\b , \text{ for all } r \in [a(s),t].\]
  Arguing as above we obtain
  \begin{align*}
    v(t,s) 
    &\le \left(  v((a(s),s)^{-(\b-1)} + \frac{\b-1}{2} \int_{a(s)}^t h(\tau) d\tau \right)^{\frac{-1}{\b-1}} \\
    &\le \left( \frac{\b-1}{2} \int_{a(s)}^t h(\tau) d\tau \right)^{\frac{-1}{\b-1}} \to 0, 
  \end{align*}
  for $a(s) \to -\infty$. Since $0 < \left( \frac{2p(t)}{h(t)} \right)^{\frac{1}{\b}} < v(t,s) $ this implies $a(s) \ge a_1 = a_1(t,p,h)$ for all $s \le s_0$.
  
  On $[a(s),t]$ we have $\frac{h(r)}{2} v(r,s)^\b \ge p(r)$ and thus
  \begin{align*}
    v(t,s) 
    &\le v(a(s),s) - \int_{a(s)}^t h(\tau) v(\tau,s)^\b d\tau + \int_{a(s)}^t p(\tau)d\tau \\
    &\le v(a(s),s) - \int_{a(s)}^t p(\tau) d\tau \\
    &\le \sup_{r \in [a_1-1,t]} \left( \frac{2 p(r)}{h(r)}  \right)^{\frac{1}{\b}} =: R(t,p,h),\quad \forall s \le s_0.
  \end{align*}
\end{proof}

\subsection{Strictly stationary solutions (Theorem \ref{thm:nonlinear_OU} \& Theorem \ref{thm:nonlinear_OU_prop})}\label{sec:stat_OU}

In this Section we will construct strictly stationary solutions to the SPDE 
  $$dX_t = M(t,X_t)dt + B_tdW_t,$$
by letting $s \to -\infty$ in $X(t,s;\o)x$ and then selecting a strictly stationary version $\eta$ from the resulting stationary limit process by use of Proposition \ref{prop:stat_perfect}. 

We will need the following elementary Lemma.
\begin{lemma}\label{lemma:integral_shift}
  Let $s \in \R$, $t \ge s$ and $h > 0$. Then there is a $\P$-zero set $N_{s,h} \in \mcF$ such that
    \[ \int_{s+h}^{t+h} B_r dW_r (\o) = \int_s^{t} B_r dW_r (\t_h\o), \quad \forall \o \in \O\setminus N_{s,h}.\]
\end{lemma}
\begin{proof}
  By the construction of the stochastic integral, for any sequence of partitions $\D^n(s,t)$ of $[s,t]$ with $|\D^n(s,t)| \to 0$ there is a subsequence of partitions $\D^{n_k}(s,t)$ and a $\P$-zero set $\bar N_s \in \bar\mcF$ such that
    \[ \int_s^{t} B_r dW_r (\o) = \lim_{k \to \infty} \sum_{t_i,t_{i+1} \in \D^{n_k}(s,t); t_i < t_{i+1} } B_{t_i}(\o) (W_{t_{i+1} \wedge t}(\o)-W_{t_i \wedge t}(\o)), \]
  for all $\o \in \O\setminus N_s$. We can choose a $\P$-zero set $N_s \in \mcF$ such that $\bar N_s \subseteq N_s$. Hence, for $\o \in \O\setminus (N_{s+h} \cup \t_h^{-1}N_s) =: N_{s,h}$
  \begin{align*}
    &\int_{s+h}^{t+h} B_r dW_r (\o) \\
    &\hskip10pt = \lim_{n \to \infty} \sum_{t_i,t_{i+1} \in \D^n(s+h,t+h); t_i < t_{i+1} } B_{t_i}(\o) (W_{t_{i+1} \wedge (t+h)}(\o)-W_{t_i  \wedge (t+h)}(\o)) \\
    &\hskip10pt= \lim_{n \to \infty} \sum_{t_i,t_{i+1} \in \D^n(s+h,t+h); t_i < t_{i+1} } B_{t_i-h}(\t_h\o) (W_{(t_{i+1}-h)  \wedge t}(\t_h\o)-W_{(t_i-h)  \wedge t}(\t_h\o)) \\
    &\hskip10pt= \int_s^{t} B_r dW_r (\t_h\o),
  \end{align*}
  where we have selected a subsequence of partitions if necessary. 
\end{proof}

\begin{proof}[Proof of Theorem \ref{thm:nonlinear_OU}]
  Let $x \in H$ and $X(t,s;\o)x$ denote the unique $\bar\mcF_t$-adapted variational solution associated to 
  \begin{equation}\label{eqn:SPDE_2}
    dX_t = M(t,X_t)dt + B_t dW_t
  \end{equation}
  given by Theorem \ref{thm:var_ex}.

  (i): There is an $\bar\mcF_t$-adapted, $\mcF$-measurable process $u: \R \times \O \to H$ such that
  \begin{equation*}
    \lim_{s \to -\infty} X(t,s;\cdot)x = u_t,
  \end{equation*}  
  in $L^2(\O;H)$ for each $t \in \R$, independent of $x \in H$. 

  Let $s_2 \ge s_1$ and $x,y \in H$. Then
  \begin{align*}
    & \frac{d}{dt}\|X(t,s_2;\o)x- X(t,s_1;\o)y\|_H^2  \\
    &= 2\ _{V^*}\< M(t,X(t,s_2;\o)x) - M(t,X(t,s_1;\o)y),X(t,s_2;\o)x - X(t,s_1;\o)y\>_V \nonumber\\
    &\le - c \left(\|X(t,s_2;\o)x-X(t,s_1;\o)y\|_H^2\right)^\frac{\a}{2}, \nonumber
  \end{align*}
  for a.e.\ $t \ge s_2$ and $\P$-a.a.\ $\o \in \O$. We apply Lemma \ref{lemma:comp_1} (Gronwall's inequality for $\a=2$) to conclude
  \begin{equation}\label{eqn:diff}
    \|X(t,s_2;\o)x - X(t,s_1;\o)y\|_H^2 \le
    \begin{cases}
        \left( (\b-1)c (t-s_2)\right)^{-\frac{1}{\b-1}} & \text{, if } \b > 1 \\
        \|x-X(s_2,s_1;\o)y\|_H^2  e^{-c (t-s_2) } & \text{, if } \b = 1,
   \end{cases}
  \end{equation}
  where $\b := \frac{\a}{2}$, for all $t \ge s_2$. Since the right hand side is decreasing in $t$ this yields
    $$ \E \sup_{r \in [t,\infty)} \|X(r,s_2;\cdot)x - X(r,s_1;\cdot)y\|_H^2 \le
    \begin{cases}
        \left( (\b-1)c (t-s_2)\right)^{-\frac{1}{\b-1}} & \text{, if } \b > 1 \\
        \E\|x-X(s_2,s_1;\cdot)y\|_H^2  e^{-c (t-s_2) } & \text{, if } \b = 1,
   \end{cases} $$
  for all $t \ge s_2$. 

  If $\a=2$ there is a ${\d}(t) = C (f_t + 1)$ such that
    \[ 2 \ _{V^*}\< M(t,v),v\>_V + \|B_t\|_{L_2(U,H)}^2 \le -\frac{c}{2} \|v\|_H^2 + {\d}(t),\]
  for each $v \in V$ (cf. \cite[Lemma 4.3.8]{PR07}). Hence,
  \begin{align*}
    \frac{d}{dt}\E \|X(t,s;\cdot)y\|_H^2 
    &= \E \left(2\Vbk{M(t,X(t,s;\cdot)y),X(t,s;\cdot)y} + \|B_t\|_{L_2(U,H)}^2\right) \\
    &\le - \frac{c}{2} \E \|X(t,s;\cdot)y\|_H^2 + \E\d(t ),\ \text{for a.e. } t \ge s.
  \end{align*}
  Thus, by Gronwall's inequality and exponential integrability of $\E f_t$
  \begin{align*}
    \E \|X(s_2,s_1;\cdot)y\|_H^2
      &\le e^{-\frac{c}{2}(s_2-s_1)}\|y\|_H^2 + \int_{s_1}^{s_2} e^{-\frac{c}{2}(s_2-r)} \E\d(r)dr \\
      &\le e^{-\frac{c}{2}s_2} \left( e^{\frac{c}{2}s_1}\|y\|_H^2 + \int_{-\infty}^{0} e^{\frac{c}{2}r} \E\d(r)dr \right) \\
      &\le e^{-\frac{c}{2}s_2} \left( e^{\frac{c}{2}s_1}\|y\|_H^2 + C \right),
  \end{align*}
  for all $t \ge s$. We obtain 
  \begin{equation*}\begin{split}\label{eqn:diff_2}
    &\E \sup_{r \in [t,\infty)} \|X(r,s_2;\cdot)x - X(r,s_1;\cdot)y\|_H^2 \\
    &\hskip10pt\le 
    \begin{cases}
        \left( (\b-1)c (t-s_2)\right)^{-\frac{1}{\b-1}} & \text{, if } \b > 1 \\
        2 \left( e^{\frac{c}{2}s_1}\|y\|_H^2 + e^{\frac{c}{2} s_2 }\|x\|_H^2 + C \right) e^{\frac{c}{2} s_2 } e^{-c t } & \text{, if } \b = 1,
   \end{cases}
  \end{split}\end{equation*}
  for all $t \wedge 0 \ge s_2$.  Hence, $X(\cdot,s;\cdot)x$ is a Cauchy sequence in $L^2(\O;C([t,\infty);H))$ and 
    \[ u_t = \lim_{s \to -\infty} X(t,s;\cdot)x       \]
  exists as a limit in $L^2(\O;H)$ for all $t \in \R$ and $u$ is $\bar\mcF_t$-adapted. Since $X(\cdot,s;\cdot)x$ also converges in $L^2(\O;C([t,\infty);H))$ $u$ is continuous $\P$-almost surely and since $u$ is $\bar\mcF$-measurable we can choose an indistinguishable $\mcF$-measurable version of $u$.

  (ii): $u$ solves \eqref{eqn:all_time_spde}. 

  Let $t_0 \ge s_0 \ge s$. Then
    \[ X(t,s;\o)x = X(s_0,s;\o)x + \int_{s_0}^t M(r,X(r,s;\o)x) dr + \int_{s_0}^t B_rdW_r(\o), \]
  for all $t \in [s_0,t_0]$ and $\P$-a.a.\ $\o \in \O$. Using condition $(H3)$ we obtain
  \begin{align*}
    &\E e^{-C t_0}\|X(t_0,s;\cdot)x\|_H^2 \\
    &= \E e^{-C s_0}\|X(s_0,s;\cdot)x\|_H^2 - C \int_{s_0}^{t_0} \E e^{-C r} \|X(r,s;\cdot)x\|_H^2 dr \\
      &\hskip15pt + \E \int_{s_0}^{t_0} e^{-C r} \left( \ _{v^*}\< M(r,X(r,s;\cdot)x), X(r,s;\cdot)x\>_V  + \|B_r\|^2_{L_2(U,H)} \right) dr  \\ 
    &\le \E e^{-C s_0}\|X(s_0,s;\cdot)x\|_H^2 - c \E \int_{s_0}^{t_0} e^{-C r} \|X(r,s;\cdot)x\|_V^\a dr + \E \int_{s_0}^{t_0} e^{-C r} f_r dr.
  \end{align*}
  Hence,
  \begin{align*}
     c \E \int_{s_0}^{t_0} \|X(r,s;\cdot)x\|_V^\a dr 
    &\le \E e^{-C (s_0-t_0)}\|X(s_0,s;\cdot)x\|_H^2  + \E \int_{s_0}^{t_0} e^{-C (r-t_0)} f_r dr
  \end{align*}
  and the right hand side converges to $\E e^{-C (s_0-t_0)}\|u_{s_0}\|_H^2  + \E \int_{s_0}^{t_0} e^{-C (r-t_0)} f_r dr$ for $s \to -\infty$. Thus $X(\cdot,s;\cdot)x$ is bounded in $L^\a([s_0,t_0] \times \O; V)$. Since $\lim\limits_{s \to -\infty} X(\cdot,s;\cdot)x = u_\cdot$ in $L^2(\O;C([s_0,t_0];H))$, it follows
    \[  X(\cdot,s;\cdot)x \rightharpoonup u, \text{ in } L^\a([s_0,t_0] \times \O; V).\]
  In particular, $u \in L^\a(\O;L^\a_{loc}(\R;V))$. Let $s_n \to -\infty$. We observe (at least for a subsequence)
    \[  M(\cdot,X(\cdot,s_n;\cdot)x) \rightharpoonup Y, \text{ in } L^\a([s_0,t_0] \times \O; V)^*.\]
  Since $X(\cdot,s;\cdot)$ and $M$ are $\bar\mcF_t$-progressively measurable, so is $Y$. Let $v \in V$ and $\vp \in L^\infty([s_0,t_0]\times\O)$. Then
  \begin{align*}
    &\E\int_{s_0}^{t_0} \ _{V^*}\< u_t, v \>_V \vp_t dt
      = \lim_{n \to \infty} \E\int_{s_0}^{t_0} \ _{V^*}\< X(t,s_n;\cdot)x, v \>_V  \vp_t dt\\
    & = \lim_{n \to \infty} \E\int_{s_0}^{t_0} \left( \ _{V^*}\< X(s_0,s_n;\cdot)x + \int_{s_0}^t M(\tau,X(\tau,s_n;\cdot)x) d\tau + \int_{s_0}^t B_\tau dW_\tau, v \>_V \right) \vp_t dt \\
    & = \E\int_{s_0}^{t_0} \ _{V^*}\< u_{s_0} + \int_{s_0}^{t} Y_\tau + \int_{s_0}^{t} B_\tau dW_\tau, v \>_V \vp_t dt.
  \end{align*}
  Hence
   \[ u_t = u_{s_0} + \int_{s_0}^{t} Y_\tau d\tau + \int_{s_0}^t B_\tau dW_\tau, \quad \forall t \in [s_0,t_0],\ \P\text{-a.s.}\]
  By It\^o's formula
  \begin{align}\label{eqn:eta_ito}
    \E \|u_{t_0}\|_H^2  = \E \|u_{s_0}\|_H^2 + \E \int_{s_0}^{t_0} \ _{V^*}\< Y_\tau, u_\tau\>_V + \|B_\tau\|_{L_2(U,H)}^2 d\tau. 
  \end{align} 
  Let $\Phi \in L^\a([s_0,t_0] \times \O; V)$. We use the monotonicity trick
  \begin{align*}
   \ _{V^*}\< M(\tau,X(\tau,s;\o)x), X(\tau,s;\o)x \>_V 
    & \le  \ _{V^*}\< M(\tau,X(\tau,s;\o)x), \Phi_\tau \>_V \\
         &\hskip10pt+ \ _{V^*}\< M(\tau,\Phi_\tau), X(\tau,s;\o)x-\Phi_\tau \>_V
  \end{align*}
  to obtain
  \begin{align*}
    \E \|X({t_0},s;\cdot)x\|_H^2 
    &\le \E \|X(s_0,s;\cdot)x\|_H^2 + \E \int_{s_0}^{t_0} \ _{V^*}\< M(\tau,X(\tau,s;\cdot)x), \Phi_\tau \>_V \\
      &\hskip15pt + \ _{V^*}\< M(\tau,\Phi_\tau), X(\tau,s;\cdot)x-\Phi_\tau \>_V + \|B_\tau\|_{L_2(U,H)}^2 d\tau .
  \end{align*} 
  Taking $s \to -\infty$ yields
  \begin{align*}
    \E \|u_{t_0}\|_H^2 
    &\le \E \|u_{s_0}\|_H^2 \\
      &\hskip10pt + \E \int_{s_0}^{t_0} \ _{V^*}\< Y_\tau, \Phi_\tau \>_V + \ _{V^*}\< M(\tau,\Phi_\tau), u_\tau -\Phi_\tau \>_V + \|B_\tau\|_{L_2(U,H)}^2 d\tau .
  \end{align*} 
  Substracting \eqref{eqn:eta_ito} leads to
  \begin{align*}
    0 \le \E \int_{s_0}^{t_0} \ _{V^*}\< M(\tau,\Phi_\tau)-Y_\tau, u_\tau - \Phi_\tau \>_V d\tau.
  \end{align*} 
  By choosing $\Phi = u - \ve \tilde{\phi} v$ with $v \in V$ and $\tilde{\phi} \in L^{\infty}([s_0,t_0] \times \O)$, dividing by $\ve > 0$ and taking $\ve \to 0$ we obtain
  \begin{align*}
    0 = \E \int_{s_0}^{t_0} \ _{V^*}\< M(\tau,u_\tau)-Y_\tau, \tilde{\phi} v \>_V d\tau.
  \end{align*} 
  Hence $M(\tau,u_\tau) = Y_\tau$, d$t \otimes \P$-almost surely and thus $\P$-a.s.
    \[ u_t = u_{s_0} + \int_{s_0}^{t} M(\tau,u_\tau) d\tau + \int_{s_0}^t B_\tau dW_\tau, \text{ for all } t \ge s_0. \]

  (iii):  Let now $M, B$ be strictly stationary. We prove crude stationarity for $u$. We first show $X(t,s;\o)x = X(0,s-t;\t_t \o)x$ for all $t \ge s$, $\P$-almost surely. Let $h > 0, t \ge s$ and $\bar{X}_h(t)(\o) := X(t-h,s-h;\t_h \o)x$. Then for $\P$-a.a. $\o \in \O$ (with zero set possibly depending on $s,h,x$)
  \begin{align*}
    \bar{X}_h(t)(\o) 
    &= X(t-h,s-h;\t_h \o)x \\
    &= x + \int_{s-h}^{t-h} M(r,X(r,s-h;\t_h \o)x;\t_h \o)dr + \left( \int_{s-h}^{t-h} B_r dW_r\right) (\t_h \o) \\
    &= x + \int_{s-h}^{t-h} M(r+h,X(r,s-h;\t_h \o)x;\o)dr + \left( \int_s^t B_r dW_r\right) (\o) \\
    &= x + \int_s^t M(r,\bar{X}_h(r)(\o);\o)dr + \left( \int_s^t B_r dW_r \right) (\o).
  \end{align*}
  Hence $X(t-h,s-h;\t_h \o)x = X(t,s;\o)x$, $\P$-almost surely. In particular 
  \begin{equation}\label{eqn:X-shift}
    X(0,s-t;\t_t \o)x = X(t,s;\o)x,
  \end{equation}
  $\P$-almost surely (with zero set possibly depending on $t,s,x$). For an arbitrary sequence $s_n \to -\infty$ there exists a subsequence (again denoted by $s_n$) such that $X(t,s_n;\cdot)x \to u_t$ and $X(0,s_n-t;\cdot)x \to u_0$ $\P$-almost surely. Thus passing to the limit in \eqref{eqn:X-shift} yields
   \[ u_0(\t_t \o) = u_t(\o),\]
  $\P$-almost surely (with zero set possibly depending on $t$).

  Since $u_\cdot \in L^\a(\O;L^\a_{loc}(\R;V))$, in particular $u_\cdot(\o) \in L^\a_{loc}(\R;V)$ for almost all $\o \in \O$ and since $u$ is $\mcF$-measurable, we now use Proposition \ref{prop:stat_perfect} to deduce the existence of an indistinguishable, $\mcF$-measurable, $\bar\mcF_t$-adapted, strictly stationary, continuous process $\tilde{u}$ such that $\tilde{u}_\cdot(\o) \in L^\a_{loc}(\R;V)$ for all $\o \in \O$.

  We can now proceed to prove \eqref{eqn:u_V_bound}. Let $\eta \le 1$ and note that by $(H2')$
    $$ \dualdel{V}{M(t,v)}{v} + \|B_t\|_{L_2(U,H)}^2 \le -\frac{c}{2}\|v\|_V^\a + C(1+f_t),\quad \forall v\in V.$$
  By It\^o's formula and the product rule we obtain
  \begin{align*}
    \E \|X(t,s;\cdot)x\|_H^2e^{\eta t} 
    &\le \E \|X(s_0,s;\cdot)x\|_H^2e^{\eta s_0} - \frac{c}{2} \E \int_{s_0}^t e^{\eta r} \|X(r,s;\cdot)x\|_V^\a dr  \\
    &\hskip12pt + \E \int_{s_0}^t e^{\eta r} C(1+f_r) dr + \eta \E \int_{s_0}^t e^{\eta r} \|X(r,s;\cdot)x\|_H^2 dr \\
    &\le \E \|X(s_0,s;\cdot)x\|_H^2e^{\eta s_0} - \left(\frac{c}{2} -\eta C \right) \E \int_{s_0}^t e^{\eta r} \|X(r,s;\cdot)x\|_V^\a dr \\
    &\hskip12pt + \E \int_{s_0}^t e^{\eta r} C(1+f_r) dr,\quad \forall s \le s_0 \le t.
  \end{align*}
  Letting $s \to -\infty$ we conclude
  \begin{align*}
    \E \|u_t\|_H^2e^{\eta t} + \left(\frac{c}{2}-\eta C \right) \E \int_{s_0}^t e^{\eta r} \|u_r\|_V^\a dr
    &\le \E \|u_{s_0}\|_H^2 e^{\eta s_0} + \E \int_{s_0}^t e^{\eta r} C(1+f_r) dr. 
  \end{align*}
  Stationarity of $u$ implies
  \begin{align*}
    \left(\frac{c}{2}-\eta C \right) \E \int_{-\infty}^t e^{\eta r} \|u_r\|_V^\a dr
    &\le  \E \int_{-\infty}^t e^{\eta r} C(1+f_r) dr. 
  \end{align*}
\end{proof}

\begin{proof}[Proof of Theorem \ref{thm:nonlinear_OU_prop}]
  (i): Since $\|u_r(\o)\|_H^k = \|u_0(\t_r \o)\|_H^k$ and $\t$ is an ergodic dynamical system on $(\O,\mcF,\P)$, \eqref{eqn:OU_bound} follows from Birkhoff's ergodic theorem as soon as we have shown $\|u_0\|_H^k \in L^1(\O)$. 
  By $(H2')$ there are $C, c > 0$  such that
    \[ 2 \ _{V^*}\<M(t,v),v\>_V + \|B_t\|_{L_2(U,H)}^2 \le - c \|v\|_H^\a + C + f_t.\]
  Application of It\^o's formula to $\|\cdot\|_H^k$ yields (cf.\ \cite[Lemma 2.2]{LR10})
  \begin{align*}
    \frac{d}{dt} \E \|X(t,s;\cdot)0\|_H^k 
    \le - \frac{c k}{4} \left(\E \|X(t,s;\cdot)0\|_H^k\right)^\frac{k+\a-2}{k}  + C(k)  \E (1 + f_t)^\frac{k-2+\a}{\a},  
  \end{align*}
  for a.e.\ $t \ge s$. 
  If $\a > 2$ then by Lemma \ref{lemma:comp_superlinear} 
    $$\E \|X(0,s;\cdot)0\|_H^k \le C < \infty,$$
  for $s$ small enough and for some constant $C > 0$. 
  For $\a =2$, by Gronwall's inequality
  \begin{align*}
    &\E \|X(0,s;\cdot)0\|_H^k  \le C(k)  \E \int_{s}^{0} e^{r \frac{c k}{4}} (C+ f_r)^\frac{k}{2} dr \le C < \infty.
  \end{align*}  
  Fatou's Lemma  yields
    \[ \E \|u_0\|_H^k \le \liminf_{n \to \infty} \E \|X(0,s_n;\cdot)0\|_H^k \le C < \infty\]
  and thereby \eqref{eqn:OU_bound}.

  (ii): By similar calculations as for \eqref{eqn:OU_bound} and by using Burkholder's inequality we obtain 
      $$\E \sup_{t \in [0,1]} \|u_t\|_H^k \le C\left(\E \|u_0\|^k_H + \E \int_0^1 f_r^\frac{k-2+\a}{\a} dr \right)< \infty,$$
      i.e.\ $\sup_{t \in [0,1]} \|u_0(\t_t\cdot)\|_H^k \in L^1(\O)$. By the dichotomy of linear growth (cf. \cite[Proposition 4.1.3.]{A98}) this implies
    \[ \limsup_{t \to \pm\infty} \frac{\|u_0(\t_t \o)\|_H^k}{|t|} = 0,\]
    on a $\t$-invariant set of full $\P$-measure.
\end{proof}

\subsection{Generation of stochastic flows (Theorem \ref{thm:generation})}\label{sec:generation}

\begin{proof}[Proof of Theorem \ref{thm:generation}]
  (i): We consider \eqref{eqn:transformed_spde} as an $\o$-wise random PDE. This viewpoint will be used to define the associated stochastic flow. In order to apply variational methods, we first use another transformation to cancel the linear term $\mu z_r Z(r,s;\o)x$ appearing in \eqref{eqn:transformed_spde}. Let $k(s,t;\o) := e^{- \mu \int_s^t z_r(\o) dr}$. Then $\tilde{Z}(t,s;\o)x := k(s,t;\o)Z(t,s;\o)x$ satisfies
    \[ \tilde{Z}(t,s;\o)x = x + \int_s^t k(s,r;\o) A_\o(r,k(s,r;\o)^{-1} \tilde{Z}(r,s;\o)x) dr, \quad \forall t \ge s, \]
  if and only if $Z(t,s;\o)x$ satisfies \eqref{eqn:transformed_spde}. In order to obtain the existence of a unique solution to \eqref{eqn:transformed_spde} for each fixed $(\o,s) \in \O \times \R$, we thus need to verify the assumptions $(H1)$--$(H4)$  (cf.\ Appendix \ref{app:var_spde}) for $(t,v) \mapsto k(s,t;\o) A_\o(t,k(s,t;\o)^{-1} v)$. These properties follow immediately from the respective properties for $A_\o(t,v)$. We will check $(H1)$--$(H4)$ for $A_\o(t,v)$ on each bounded interval $[S,T] \subseteq \R$ and for each fix $\o \in \O$. For ease of notation we suppress the $\o$-dependency of the coefficients occurring in the following calculations. $(H1)$ is immediate from $(A1)$ for $A$.

  $(H2)$: For $v_1,v_2 \in V$, $\o \in \O$ and $t \in [S,T]$ such that $u_t(\o) \in V$:
    \begin{align*}
    &2{  }_{V^*}\< A_\o(t,v_1)- A_\o(t,v_2), v_1-v_2\>_V \\
    &= \mu_t^2 2{  }_{V^*}\<  A \left(t, \mu_t^{-1} (v_1 + u_t)\right)- A \left(t, \mu_t^{-1} (v_2 + u_t) \right), 
         \mu_t^{-1}(v_1 + u_t) -\mu_t^{-1} (v_2 + u_t)\>_V \\
    &\le C(t) \|v_1 - v_2\|_H^2. 
    \end{align*} 
    For $t \in \R$ such that $u_t(\o) \not\in V$ the same calculation holds. This yields $(H2)$ on $[S,T]$ for $A_\o(t,v)$ by local boundedness of $C(t)$.

  $(H3)$: For $v \in V$, $\o \in \O$ and $t \in \R$ such that $u_t(\o) \in V$: 
  \begin{align}\label{eqn:H3-1}
    2{  }_{V^*}\< A_\o(t,v), v \>_V  
    &= 2 \mu_t^2{  }_{V^*}\< A \left(t, \mu_t^{-1} (v + u_t)\right),\mu_t^{-1}(v+ u_t) \>_V \nonumber\\
      &\hskip10pt -2 \mu_t {  }_{V^*}\< A \left(t, \mu_t^{-1} (v + u_t) \right),u_t \>_V \nonumber\\
      &\hskip10pt + 2 \mu z_t {  }_{V^*}\< u_t , v \>_V + 2 {  }_{V^*}\< M(u_t), v \>_V \nonumber\\
    &\le C(t) \|v + u_t\|_H^2 - c(t) \mu_t^{2-\a} \|v + u_t\|_V^\a + \mu_t^2 f(t) \\
      &\hskip10pt + 2 \mu_t \|A \left(t, \mu_t^{-1} (v + u_t) \right)\|_{V^*} \|u_t\|_V \nonumber\\
      &\hskip10pt + 2 \mu z_t {  }_{V^*}\<u_t,v\>_V + 2 {  }_{V^*}\< M(u_t), v \>_V. \nonumber
  \end{align}
  Using Young's inequality for all $\ve_1,\ve_2,\ve_3 >0$ and some $C_{\ve_1},C_{\ve_2},C_{\ve_3}$ we obtain
  \begin{align*}
    &2 \mu_t \|A \left(t, \mu_t^{-1} (v + u_t) \right)\|_{V^*} \|u_t\|_V  \\
    &\le \ve_1 \mu_t^2 \|A \left(t, \mu_t^{-1} (v + u_t) \right)\|_{V^*}^{\frac{\a}{\a-1}} + C_{\ve_1} \mu_t^{2-\a} \|u_t\|_V^\a &&\\
    &\le \ve_1 C(t) \mu_t^{2-\a}  \|v + u_t\|_V^\alpha + \ve_1 \mu_t^2 f(t) + C_{\ve_1} \mu_t^{2-\a} \|u_t\|_V^\a
  \end{align*}
  and
  \begin{align*}
  2  {  }_{V^*}\< M(u_t), v \>_V
  &\le C_{\ve_2} 2 \mu_t^{\frac{2-\a}{1-\a}}\|M(u_t)\|_{V^*}^{\frac{\a}{\a-1}} + 2 \ve_2 \mu_t^{2-\a} \| v\|_V^\a \\
  &\le C_{\ve_2} 2 \mu_t^{\frac{2-\a}{1-\a}} \Big( C_{M} \|u_t\|_V^\a + f_M \Big) + 2 \ve_2 \mu_t^{2-\a} \| v\|_V^\a, &&
  \end{align*}
  as well as
  \begin{align*}
  2 \mu z_t {  }_{V^*}\<u_t,v\>_V 
  \le C_{\ve_3} \mu^2 z_t^2 \|u_t\|_H^2 +  \ve_3\|v\|_H^2. &&
  \end{align*}
  Using this in \eqref{eqn:H3-1} yields
  \begin{align*}
      2{  }_{V^*}\< A_\o(t,v), v \>_V  
    &\le C(t) \|v + u_t\|_H^2 - (c(t)-\ve_1 C(t)) \mu_t^{2-\a} \|v + u_t\|_V^\a + \mu_t^2 f(t) \\
      &\hskip10pt + \ve_1 \mu_t^2 f(t) + C_{\ve_1} \mu_t^{2-\a} \|u_t\|_V^\a + C_{\ve_3} \mu^2 z_t^2 \|u_t\|_H^2 + \ve_3 \|v\|_H^2 \nonumber\\
      &\hskip10pt + C_{\ve_2} 2 \mu_t^{\frac{2-\a}{1-\a}} \big( C_{M} \|u_t\|_V^\a + f_M \big)+ 2 \ve_2 \mu_t^{2-\a} \| v\|_V^\a  .\nonumber
  \end{align*}
  Using
    \[ \|v+u_t(\o)\|_V^\a \ge  2^{1-\a}\|v\|_V^\a - \|u_t(\o)\|_V^\a\]
  we obtain (for $\ve_1$ small enough):
  \begin{align*}
    2{  }_{V^*}\< A_\o(t,v), v \>_V 
    &\le (2C(t) + \ve_3) \|v\|_H^2 - \Big(2^{1-\a}c(t)- \ve_1 2^{1-\a} C(t) - 2 \ve_2 \Big) \mu_t^{2-\a} \|v\|_V^\a  \\
      &\hskip10pt + \mu_t^2 f(t) + \ve_1 \mu_t^2 f(t) + \mu_t^{2-\a}\Big(C_{\ve_1}  + c(t) - \ve_1 C(t)\Big)\|u_t\|_V^\a  \nonumber\\
        &\hskip10pt + C_{\ve_2} 2 \mu_t^{\frac{2-\a}{1-\a}} \Big( C_{M} \|u_t\|_V^\a + f_M \Big) + \Big(2C(t) + C_{\ve_3} \mu^2 z_t^2 \Big) \|u_t\|_H^2 .\nonumber
  \end{align*}
  Now choosing $\ve_1,\ve_2$ small enough yields
  \begin{align}\label{eqn:transf_coerc}
    \ _{V^*}\< A_\o(t,v), v \>_V  &\le \tilde{C}(t) \|v\|_H^2 - \tilde{c}(t) \|v\|_V^\a + \tilde{f}(t),
  \end{align}
  with 
  \begin{align*}
    \tilde{f}(t) &:=  \mu_t^2 f(t) + \ve_1 \mu_t^2 f(t) + \mu_t^{2-\a}\Big(C_{\ve_1}  + c(t) - \ve_1 C(t)\Big)\|u_t\|_V^\a  \nonumber\\
        &\hskip10pt + C_{\ve_2} 2 \mu_t^{\frac{2-\a}{1-\a}} \Big( C_{M} \|u_t\|_V^\a + f_M \Big) + \Big(2C(t) + C_{\ve_3} \mu^2 z_t^2 \Big) \|u_t\|_H^2.
  \end{align*}
  By pathwise continuity of $u_t, \mu_t, z_t$, local boundedness of $c(t),C(t)$ and since $u_\cdot(\o) \in L^\a_{loc}(\R;V)$ for {\it all} $\o \in \O$ we have $\td f \in L^1_{loc}(\R)$. By choosing $\ve_1, \ve_2$ small enough we obtain $(H3)$ on $[S,T] \subseteq \R$. For $t \in \R$ such that $u_t(\o) \not\in V$ the same calculation proves $(H3)$.

  $(H4)$: For $v \in V$, $\o \in \O$ and $t \in \R$ such that $u_t(\o) \in V$: 
  \begin{align}
    \|A_\o(t,v)\|_{V^*}^{\frac{\a}{\a-1}}
    &=\| \mu_t A \left(t, \mu_t^{-1} (v + u_t) \right) + \mu u_t z_t + M(u_t) \|_{V^*}^{\frac{\a}{\a-1}} \nonumber \\
    &\le C \left( \|\mu_t A \left(t, \mu_t^{-1} (v + u_t) \right)\|_{V^*}^{\frac{\a}{\a-1}} + \left( \mu |z_t| \|u_t \|_{V^*} \right)^{\frac{\a}{\a-1}} + \| M(u_t) \|_{V^*}^{\frac{\a}{\a-1}} \right)   \nonumber\\
    &\le C \mu_t^{\frac{\a}{\a-1}} \Big( C(t) \mu_t^{-\a} \|v + u_t\|_V^\a + f(t) \Big) + C \left( \mu z_t \|u_t \|_{V^*} \right)^{\frac{\a}{\a-1}} \nonumber\\
    &\hskip15pt + C C_{M}(\|u_t\|_V^\a + 1) \nonumber\\
    &\le \td C(t) \|v\|_V^\a + \tilde{f}_t,  \nonumber
  \end{align}
  with 
  \begin{align*}
    \tilde{f}(t) 
      &:= C \mu_t^{\frac{\a}{\a-1}} f(t) + C C(t) \mu_t^{\frac{\a}{\a-1}-\a} \|u_t\|_V^\a + C \left( \mu z_t \|u_t \|_{V^*} \right)^{\frac{\a}{\a-1}} + C C_{M}(\|u_t\|_V^\a + 1).
  \end{align*}
  As before this yields $(H4)$ on $[S,T]$. For $t \in \R$ such that $u_t \not\in V$ the same calculation proves $(H4)$.
  
  Hence $(H1)$--$(H4)$ are satisfied for $A_\o$ for each $\o \in \O$ and by Theorem \ref{thm:var_ex} (applied to the deterministic situation) we obtain the existence of a unique solution 
  \begin{equation*}\label{eqn:def_transf_flow}
    Z(\cdot,s;\o)x \in L^\a_{loc}([s,\infty);V) \cap C([s,\infty);H) 
  \end{equation*}
  to \eqref{eqn:transformed_spde} for all $(s,\o,x) \in \R \times \O \times H$. By uniqueness for \eqref{eqn:transformed_spde} we have the flow property
  \begin{equation*}\label{eqn:flow_property}
    Z(t,s;\o)x = Z(t,r;\o)Z(r,s;\o)x.
  \end{equation*}
  By Proposition \ref{prop:def_conj_flow} the family of maps given by
  \begin{equation*}\label{def:flow}
    S(t,s;\o)x := T(t,\o) \circ Z(t,s;\o) \circ T^{-1}(s,\o)
  \end{equation*}
  defines a stochastic flow.

  (ii): Since $T(t,\o)y$ is continuous in $t$ locally uniformly in $y$, $t \mapsto S(t,s;\o)x$ is continuous. \cite[Proposition 4.2.10]{PR07} implies
    \[ \|Z(t,s;\o)x - Z(t,s;\o)y\|^2_H \le e^{C(t-s)}\|x-y\|_H^2, \]
  for bounded $s \le t$. Thus $x \mapsto Z(t,s;\o)x$ is continuous locally uniformly in $s,t$. Moreover,
  \begin{align*}
    \|Z(t,s_1;\o)x - Z(t,s_2;\o)x\|^2_H 
    &= \|Z(t,s_2;\o)Z(s_2,s_1;\o)x - Z(t,s_2;\o)x\|^2_H  \\
    &\le e^{C(t-s_2)}\|Z(s_2,s_1;\o)x-x\|_H^2,\ \forall s_1 < s_2, 
  \end{align*}
  which implies right-continuity of $s \mapsto Z(t,s;\o)x$. This implies right continuity of $s \mapsto S(t,s;\o)x$ and continuity of $x \mapsto S(t,s;\o)x$ locally uniformly in $s,t$.
  
  (iii): If $A$ is $(\mcB(\R)\otimes\mcB(V)\otimes\mcF,\mcB(V^*))$-measurable then measurability of $Z(t,s;\o)$, $S(t,s;\o)$ follows as in the proof of \cite[Theorem 1.1]{GLR11}. 

  The same argument proves $\bar\mcF_t$-adaptedness of $Z(t,s;\o)x$ if $A$ is $\bar\mcF_t$-adapted. Applying the transformation backwards then shows that $S(t,s;\o)x$ is a solution to \eqref{eqn:spde2}.
  
  (iv): Assume that $A(t,v;\o)$ is strictly stationary. Strict stationarity of $z_t$ and $u_t$ then implies $A_\o(t,v) = A_{\t_t\o}(0,v)$. By uniqueness for \eqref{eqn:transformed_spde} we deduce
  \begin{align*}
    Z(t,s;\o)x = Z(t-s,0;\t_s \o)x
  \end{align*}
  and $Z(t,s;\o)$ is a cocycle. Since $T(t,\o)$ is a stationary conjugation the same follows for $S(t,s;\o)$.
\end{proof}


\subsection{Existence of Random Attractors (Theorem \ref{thm:superlinear_ra})}

For $\eta: \O \to \R_+\setminus\{0\}$ let $\mcD^{\eta}$ be the set of all sets of subexponential growth of order $\eta$, i.e.\  $\{D(t,\o)\}_{t \in \R, \o \in \O} \in \mcD^{\eta}$ iff $\|D(s,\o)\|_H^2 = O(e^{\eta(\o) s})$ for $s \to -\infty$ and all $\o \in \O$. Define $B(x,r):=\{y \in H|\ \|x-y\|_H \le r\}.$ First, we will prove bounded absorption for the stochastic flows $S(t,s;\o)$, $Z(t,s;\o)$. 

\begin{proposition}[Bounded absorption]\label{prop:bdd_abs}
  Assume $(A1)$--$(A4)$ and $(V)$ with $c, C$ being independent of time $t$ and $f$ exponentially integrable. If $\a =2$, additionally assume $C < \frac{c}{4\l}$ in $(A3)$. Then there is an $\eta: \O \to \R_+\setminus\{0\}$ and a $\mcD^{\eta}$-absorbing family of bounded sets $\{F(t,\o)\}_{t \in \R, \o \in \O}$ for $S(t,s;\o)$ and $Z(t,s;\o)$. More precisely, there is a function $R: \R \times \O \to \R_+\setminus\{0\}$ such that for all $D \in \mcD^\eta$ there is an absorption time $s_0=s_0(D,t;\o) \le t$ such that
   \begin{equation}\label{eqn:abs_by_ball}
      Z(t,s;\o)D(s,\o) \subseteq B(0,R(t,\o)),
   \end{equation}    
  for all $s \le s_0$, $\P$-almost surely.
\end{proposition}
\begin{proof}
  Let $t \in \R$. From \eqref{eqn:transf_coerc} and using the assumption in case $\a=2$ we obtain
    $$  \ _{V^*}\< A_\o(t,v), v \>_V  \le - \tilde{c}(\o) \mu_t^{2-\a}(\o) \|v\|_H^2 + \tilde{f}(t,\o),$$
  for some $\td c(\o) > 0$ and 
  \begin{align*}
    \tilde{f}(t,\o) 
    &=  C\mu_t^2(\o) f(t,\o)  + C \mu_t^{2-\a}(\o)\|u_t(\o)\|_V^\a + C \mu_t^{\frac{2-\a}{1-\a}}(\o) \Big( \|u_t(\o)\|_V^\a + 1 \Big)\\
       &+ C\Big(1 + z_t^2(\o) \Big) \|u_t(\o)\|_H^2 + C\mu_t^2(\o).
  \end{align*}
  Hence,
      $$ 2\dualdel{V}{A_\o(t,v)+\mu z_t(\o)v}{v} \le \left( - \td c(\o) \mu_t^{2-\a}(\o) + \mu z_t(\o)\right) \|v\|_H^2 + \td f(t,\o),\quad \forall v\in V$$
  and for a.e.\ $t \ge s$ we obtain
  \begin{align*}
    \frac{d}{dt} \|Z(t,s;\o)x\|_H^2 
    &= \ _{V^*}\< A_{\o}(t,Z(t,s;\o)x)+\mu z_t Z(t,s;\o)x, Z(t,s;\o)x \>_V  \\
    &\le \left(- \td c(\o) \mu_t^{2-\a}(\o) + \mu z_t(\o) \right) \|Z(t,s;\o)x\|_H^2 + \td f(t,\o).
  \end{align*}
  By Birkhoff's ergodic theorem and dichotomy of linear growth \cite[Proposition 4.1.3]{A98}, there is a subset $\O_0 \subseteq \O$ of full $\P$-measure such that $\frac{1}{t-s}\int_s^t \mu z_\tau(\o) d\tau \to \mu\E[z_0] = 0$, $\lim_{t \to \pm \infty}\frac{|z_t(\o)|}{|t|} \to 0$ and $\frac{1}{t-s}\int_s^t \mu_\tau^{2-\a}(\o) d\tau \to \E[\mu_0^{2-\a}]$ for $s \to -\infty$ and all $\o \in \O_0$. Therefore, $\mu_t(\o) = e^{\mu z_t(\o)}$ is of subexponential growth and $\td f$ is exponentially integrable, $\P$-almost surely. Furthermore, there is an $s_0(\o)\le 0$ such that
    $$\frac{1}{t-s}\int_s^t \left(- \td c(\o) \mu_\tau^{2-\a}(\o) + \mu z_\tau(\o)\right)\ d\tau \le - \eta(\o),$$
  for all $s \le s_0(\o)$, $\o \in \O_0$ and some $\eta: \O \to \R\setminus\{0\}$. Let $D \in \mcD^{\eta}$, $x_s(\o) \in D(s,\o)$, i.e.\ there is a function $C: \O \to \R$ such that $\|D(s,\o)\|_H^2 e^{\eta(\o) s} \le C(\o)$ for all $s$ small enough. For some $\td s_0 = \td s_0(D,t;\o)$ using Gronwall's inequality we obtain 
  \begin{equation}\begin{split}\label{eqn:bdd_abs}
     \|Z(t,s;\o)D(s,\o)\|_H^2 
     &\le \|D(s,\o)\|_H^2 e^{-\eta(\o) (t-s)} + \int_s^{s_0} e^{- \eta(\o) (t-r)} \td f(r,\o)dr \\
      &\hskip12pt + \int_{s_0}^t e^{\int_r^t \left(- \td c(\o) \mu_\tau^{2-\a}(\o) + \mu z_\tau(\o)\right) d\tau} \td f(r,\o)dr \\
     &\le C(\o)e^{- \eta(\o) t} + \int_{-\infty}^{s_0} e^{- \eta(\o) (t-r)} \td f(r,\o)dr \\
      &\hskip12pt + \int_{s_0}^t e^{\int_r^t \left(- \td c(\o) \mu_\tau^{2-\a}(\o) + \mu z_\tau(\o)\right) d\tau} \td f(r,\o)dr \\
     &=: R(t,\o), \quad \forall s \le \td s_0,\ \P-\text{a.s.,}
  \end{split}\end{equation}
  where finiteness of the second summand follows from exponential integrability of $\td f$.
  
  Since $T(t,\o)$ is a bounded map of subexponential growth this implies bounded absorption for $S(t,s;\o)$.
\end{proof}

\begin{proof}[Proof of Theorem \ref{thm:superlinear_ra}]
  (i): Compactness of the stochastic flows $S(t,s;\o)$, $Z(t,s;\o)$ follows as in \cite{G11d}.

  (ii): We prove that $Z(t,s;\o)$ is $\mcD$-asymptotically compact. By Proposition \ref{prop:bdd_abs} there is an $\eta: \O \to \R\setminus\{0\}$ and a bounded $\mcD^\eta$-absorbing set $F$. Let
    \[ K(t,\o) := \overline{Z(t,t-1;\o)F(t-1,\o)},\ \forall t \in \R,\ \o \in \O. \]
  Since $F(t-1,\o)$ is a bounded set and $Z(t,s;\o)$ is a compact flow, $K(t,\o)$ is compact. Furthermore, $K(t,\o)$ is $\mcD^\eta$-absorbing:
  \begin{align*}
    Z(t,s;\o)D(s,\o) 
    &= Z(t,t-1;\o)Z(t-1,s;\o)D(s,\o) \\
    &\subseteq Z(t,t-1;\o)F(t-1,\o) \subseteq K(t,\o),
  \end{align*}
  for all $s \le s_0$, $\P$-almost surely. By Theorem \ref{thm:suff_cond_attr} and Theorem \ref{thm:conj_attractor} this yields the existence of random $\mcD^\eta$-attractors $\mcA_Z,\mcA_S$ for $Z(t,s;\o)$, $S(t,s;\o)$ respectively. Restricting the domain of attraction of $\mcA_Z, \mcA_S$ to the system of all tempered sets $\mcD$ thus yields the claim. 

  (iii): Let now $A$ be $(\mcB(\R)\otimes\mcB(V)\otimes\mcF,\mcB(V^*))$-measurable and $c,C$ in $(A3)$ be $(\mcF,\mcB(\R))$-measurable. Then $\eta$ in  Proposition \ref{prop:bdd_abs} can be chosen $(\mcF,\mcB(\R))$-measurable and we define 
    $$\mcD_0 :=\left\{ \{D^n(t,\o) := B(0,e^{\eta(\o) |t|}n)\}_{t \in \R,\o \in \O}|\ n \in \N \right\}.$$
  Then
    $$ \mcA_Z(t,\o) = \overline{\bigcup_{n\in\N}\O(D^n,t;\o)}$$
  and Theorem \ref{thm:suff_cond_attr} implies measurability of the random $\mcD^{\eta}$-attractors. 

  (iv): Now assume $A$ to be strictly stationary. Then $Z(t,s;\o)$ and $S(t,s;\o)$ are cocycles and $K$ is a $\mcD^b$-absorbing compact set. Hence, there exist measurable, strictly stationary random $\mcD^b$-attractors for $Z(t,s;\o)$ and $S(t,s;\o)$ by Theorem \ref{thm:suff_cond_attr}.
\end{proof}

\subsection{Stable equilibria (Theorem \ref{thm:equil})}\label{sec:stationary_soln}

We will need the following
\begin{lemma}\label{lemma:integral_div}
  Let $\b$ be a real-valued Brownian motion and $q,t \in \R$. Then
    \[ \int_s^t e^{q (\b_r-\b_t)} dr \to \infty,\quad \P\text{-a.s. for } s \to -\infty.\]
\end{lemma}
\begin{proof}
  Since $s \mapsto \int_s^t e^{q (\b_r-\b_t)} dr$ is increasing as $s$ decreases, it is enough to prove divergence in probability. Observe
    \[ \int_s^t e^{q (\b_r(\o)-\b_t(\o))} dr = \int_{s-t}^0 e^{q (\b_{r+t}(\o)-\b_t(\o))} dr = \int_{s-t}^0 e^{q \b_{r}(\t_t\o)} dr.\]
  The claim then easily follows from the L\'evy arcsine law.
\end{proof}

\begin{lemma}\label{lemma:contractivity}
  Assume $(V)$ and let $A$ be $\bar\mcF_t$-adapted and satisfy $(A1)$--$(A4)$, $(A2')$. Then there is a set $\O_0 \subseteq \O$ of full $\P$-measure such that
  \begin{align*}
     &\|S(t,s_2;\o)x - S(t,s_1;\o)y\|_H^2 \le 
        \left( \left(\frac{\a}{2}-1 \right)\l(\o) \int_{s_2}^t e^{(\a-2)\mu (\b_r(\o)-\b_t(\o))}  dr \right)^{-\frac{2}{\a-2}}
  \end{align*}
  for all $t \wedge 0 \ge s_2 \ge s_1$, $\o \in \O_0$. In particular, there is an $\mcF$-measurable continuous process $\eta: \R \times \O \to H$ such that
    $$\lim \limits_{s \rightarrow -\infty}S(t,s;\o)x = \eta(t,\o),$$
  for all $x \in H$, $t\in\R$, $\P$-almost surely.
\end{lemma}
\begin{proof}
  We consider an alternative transformation of the SPDE \eqref{eqn:spde2}. Since $S(t,s;\o)x$ is a solution of \eqref{eqn:spde2}, 
    $$\td Z(t,s;\o)x := \td\mu_t(\o)S(t,s;\o)x,$$
  with $\td\mu_t(\o) := e^{-\mu \b_t(\o)}$ satisfies
  \begin{equation*}\begin{split}
    \td Z(t_2,s;\o)x &= \td Z(t_1,s;\o)x+\int_{t_1}^{t_2} \td\mu_r(\o) A(r,\td\mu_r(\o)^{-1}\td Z(r,s;\o)x)dr + \int_{t_1}^{t_2} \td\mu_r \circ dW_r(\o) \\
    \td Z(s,s;\o)x &= e^{-\mu\b_s(\o)}x, 
  \end{split}\end{equation*} 
  $\P$-a.s.\ for all $t_2 \ge t_1 \ge s$ (with $\P$-zero set possibly depending on $x$, $s$). By $(A2')$ we obtain
  \begin{align*}
    \frac{d}{dt} \|\td Z(t,s_2;\o)x- \td Z(t,s_1;\o)y\|_H^2 
    &\le - \l(\o) \td\mu_t(\o)^{2-\a} \left( \|\td Z(t,s_2;\o)x- \td Z(t,s_1;\o)y\|_H^2 \right)^\frac{\a}{2},
  \end{align*}
  for a.e.\ $t \ge s_2 \ge s_1$ and by Lemma \ref{lemma:comp_1} to conclude
  \begin{align*}
    \|\td Z(t,s_2;\o)x- \td Z(t,s_1;\o)y\|_H^2 
    &\le \left( \left( \frac{\a}{2}-1 \right)\l(\o) \int_{s_2}^t \td\mu_r(\o)^{2-\a}dr \right)^{-\frac{2}{\a-2}},
  \end{align*}
  for all $t \ge s_2$. Hence,
  \begin{equation}\label{eqn:conv_bound}\begin{split}
     & \|S(t,s_2;\o)x- S(t,s_1;\o)y\|_H^2 \\
     &\hskip20pt \le \left( \left( \frac{\a}{2}-1 \right)\l(\o) \int_{s_2}^t e^{(\a-2)\mu (\b_r(\o)-\b_t(\o)) } dr \right)^{-\frac{2}{\a-2}}, \quad \forall t \ge s_2 \ge s_1,
  \end{split}\end{equation}
 $\P-$almost surely.
  Since both sides of the inequality are continuous in $x,t$ and right-continuous in $s_1,s_2$ the claim follows on a $\P$-zero set independent of $x,s_1,s_2,t$. By Lemma \ref{lemma:contractivity} the right hand side converges to $0$ for $s_2 \to \infty$ and thus $S(t,s;\o)x$ is a Cauchy sequence with limit $\eta(t,\o)$ independent of $x$. 

  Letting $s_1 \to - \infty$ in \eqref{eqn:conv_bound} then yields
  \begin{equation*}
     \|S(t,s;\o)x-\eta(t,\o)\|_H^2 
    \le \left( \left( \frac{\a}{2}-1 \right)\l(\o) \int_{s}^t e^{(\a-2)\mu (\b_r(\o)-\b_t(\o)) } dr \right)^{-\frac{2}{\a-2}},\ \forall t \ge s,
  \end{equation*}
  $\P$-almost surely. In particular, the convergence holds locally uniformly in $t$, which implies continuity of $\eta$.
\end{proof}

\begin{proof}[Proof of Theorem \ref{thm:equil}:]
  By Lemma \ref{lemma:contractivity} we may define
    \[  \mcA(t,\o) := \{\eta(t,\o)\}. \]
  We shall show that this defines a global random $\mcD^g$-attractor for the stochastic flow $S(t,s;\o)$. Since $\eta$ is $\mcF$-measurable, $\mcA(t,\o)$ is a random compact set. It remains to check the invariance and attraction properties for $\mcA(t,\o)$. The continuity of $x \mapsto S(t,s;\o)x$ and the flow property imply that
  \begin{align*}
    S(t,s;\o)\mcA(s,\o)
      &= \left\{ S(t,s;\o) \lim_{r \rightarrow -\infty}S(s,r;\o)x \right\} = \left\{ \lim_{r \rightarrow -\infty} S(t,r;\o)x \right\} = \mcA(t,\o),
  \end{align*}
  for all $t \ge s$, $\o \in \O_0$. For any family of subsets $D \in \mcD^g$ we have
  \begin{align*}
      d(S(t,s;\o)D(s,\o),\mcA(t,\o))
          &= \sup_{x \in D(s,\o)} \|S(t,s,\o)x - \eta(t,\o)\|_H \rightarrow 0, \text{ for } t\rightarrow \infty,
  \end{align*}
  i.e.\ $\mcA(t,\o)$ is $\mcD^g$-attracting. Therefore, $\mcA$ is a random $\mcD^g$-attractor for $S(t,s;\o)$. The bound on the speed of attraction follows immediately from the respective bound in Lemma \ref{lemma:contractivity}.

  If $A$ is strictly stationary then $S(t,s;\o)$ is a cocycle. Hence
    $$\eta(t,\o) = \lim_{s \to -\infty}S(t,s;\o)x = \lim_{s \to -\infty}S(0,s-t;\t_t\o)x = \eta(0,\t_t\o),$$
  on a $\P$-zero set possibly depending on $t$. By \cite[Proposition 2.8]{L01} we can choose a strictly stationary indistinguishable version of $\eta$.
\end{proof}

\subsection{Random attractors for strongly mixing RDS (Remark \ref{rmk:interval=point} \& Theorem \ref{thm:collapse})}

\begin{proof}[Proof of Remark \ref{rmk:interval=point}:]
  By \cite[Proposition 1]{CS04} we know that for $x \le y$, $x,y \in H$ we have 
    $$\|\vp(t,\t_{-t}\cdot)x-\vp(t,\t_{-t}\cdot)y\|_H \xrightarrow{p} 0,\quad \text{for } t \to \infty,$$
  where $\xrightarrow{p}$ denotes convergence in probability. Since $x_+ \ge 0,x$ this implies
  \begin{equation}\label{eqn:0-conv}
    \|\vp(t,\t_{-t}\cdot)x-\vp(t,\t_{-t}\cdot)0\|_H \xrightarrow{p} 0,\quad \text{for } t \to \infty,
  \end{equation}
  for all $x \in H$.
  Let $\mcA$ be a weak point random attractor and $[x,y] \subseteq H$ be an interval in $H$. Since $\vp$ is order-preserving, 
    $$\vp(t;\o)[x,y] \subseteq [\vp(t;\o)x,\vp(t;\o)y],$$
  and thus 
    $$\diam(\vp(t;\o)[x,y]) \le \diam([\vp(t;\o)x,\vp(t;\o)y]).$$
  Since $H_+$ is normal we have 
    $$\diam([\vp(t;\o)x,\vp(t;\o)y]) \le 2c \|\vp(t;\o)x-\vp(t;\o)y\|_H \xrightarrow{p} 0,$$
  by \eqref{eqn:0-conv}. Hence, $\diam(\vp(t;\t_{-t}\o)[x,y]) \xrightarrow{p} 0$ and $d(\vp(t,\t_{-t}\o)x,\mcA(\o)) \xrightarrow{p} 0$ for all $x \in H$, which implies
    $$ d(\vp(t;\t_{-t}\o)[x,y],\mcA(\o)) \xrightarrow{p} 0.$$
\end{proof}

\begin{proof}[Proof of Theorem \ref{thm:collapse}:]
  The proof relies on a modification of the proof of \cite[Proposition 2]{CS04}. Since $\mu(V)=1$, we can find a sequence of compact sets $K_n     \subseteq V$ such that $\mu(K_n) \ge 1-2^{-n-2}$. By the assumptions on $V$ we can choose $f_n, g_n \in H$ such that $K_n \subseteq [f_n,g_n]$. Thus $\mu([f_n,g_n]) \ge 1-2^{-n-2}$. Now we can proceed as in  \cite[Proposition 2]{CS04} to prove the existence of an equilibrium $v(\o)$. I.e. $v: \O \to H$ is an $\F_{-\infty}$-measurable random variable satisfying $\vp(t,\o)v(\o)=v(\t_t\o)$. 

    Without loss of generality we may assume that $0 \in K_n$, thus $0 \in [f_n,g_n]$ for all $n \in \N$. As in \cite[Proposition 2]{CS04} we may then prove that $\vp(t,\t_{-t}\cdot)0 \xrightarrow{p} v$. By \eqref{eqn:0-conv} this implies
    \begin{equation}\label{eqn:weak_attraction}
     \vp(t,\t_{-t}\cdot)x \xrightarrow{p} v,\quad \text{for } t \to \infty.
    \end{equation}
    Hence, $v$ is a weak point random attractor. 

    By \eqref{eqn:weak_attraction} we have $\mcL(\vp(t,\t_{-t}\cdot)x) \to \mcL(v)$ weakly. Hence, $\mu = \mcL(v)$. Since $\mu(V)=1$ this implies $v(\o) \in V$, $\P$-almost surely. 
    
    Let now $\td \mcA$ be another weak point/interval random attractor. For each $\ve > 0$ there is a compact set $K_\ve \subseteq V$ and $\O_\ve \subseteq \O$ such that $v(\o) \subseteq K_\ve$ for all $\o \in \O_\ve$ and $\P[\O_\ve] \ge 1 -\ve$.  Since $V_+$ is $H$-solid, all compact sets in $V$ are attracted by $\td \mcA$ in probability. Thus, 
    \begin{align*}
       d(v(\o),\td \mcA(\o)) 
       &= d(\vp(t;\t_{-t}\o)v(\t_{-t}\o),\td \mcA(\o)) \\
       &\le d(\vp(t;\t_{-t}\o)K_\ve,\td \mcA(\o)),\quad \forall \o \in \t_t\O_\ve.
    \end{align*}
    Let $\d > 0$. Then,
      $$\P[d(v,\td \mcA) > \d] \le  \P[d(\vp(t;\t_{-t}\cdot)K_\ve,\td \mcA) > \d] + \ve \le 2 \ve, $$
    for all $t$ large enough. Hence, $v \in \td \mcA$, $\P$-almost surely and $\mcA$ is shown to be the minimal weak point/interval attractor.    
\end{proof}

\appendix
\section{Existence of unique solutions to monotone SPDE}\label{app:var_spde}
We recall the variational approach to monotone SPDE (cf.\ \cite{KR79,P75,PR07}). Let $V \subseteq H \subseteq V^*$ be a Gelfand triple, $(\O,\mcG,\{\mcG_t\}_{t \in [0,T]},\P)$ be a normal filtered probability space and $W$ be a cylindrical Wiener process on some separable Hilbert space $U$. Further assume that
  \[ A: [0,T] \times V \times \O \to V^*,\  B: [0,T] \times V \times \O \to L_2(U,H) \]
are $\mcG_t$-progressively measurable. We extend $A, B$ by $0$ to all of $[0,T] \times H \times \O$.

\begin{definition}\label{def:prob_soln}
  A continuous, $H$-valued, $\mcG_t$-adapted process $\{X_t\}_{t \in [0,T]}$ is called a solution to 
  \begin{equation}\label{eqn:spde3}
    dX_t = A(t,X_t)dt + B(t,X_t)dW_t
  \end{equation}
  if $X \in L^{\a}([0,T]\times\O;V)$ and $\P$-a.s.
    \[ X_t = X_0 + \int_0^t A(r,X_r)dr + \int_0^t B(r,X_r)dW_r,\ \forall t \in [0,T]. \]
\end{definition}

We assume that there are $\a > 1$, $c,C > 0$ and $f \in L^1([0,T]\times\O)$ positive and $\mcG_t$-adapted such that 
\begin{enumerate}
 \item[(H1):] (Hemicontinuity) The map $s \mapsto \Vbk{A(t,v_1+sv_2),v}$ is continuous on $\R$,
 \item[(H2):] (Monotonicity) 
    \[ 2 \Vbk{A(t,v_1)-A(t,v_2),v_1-v_2} + \|B(t,v_1)-B(t,v_2)\|_{L_2(U,H)}^2 \le C\|v_1-v_2\|_H^2,\]
 \item[(H3):] (Coercivity)
    \[ 2 \Vbk{A(t,v),v} + \|B(t,v)\|_{L_2(U,H)}^2 + c\|v\|_V^\a \le f_t + C\|v\|_H^2, \]
 \item[(H4):] (Growth)
  \begin{align*}
    \left\| A(t,v) \right\|_{V^*}^\frac{\a}{\a-1} \le C\|v\|_V^\a+f_t,
  \end{align*}   
\end{enumerate}
for all $v_1, v_2, v \in V$ and $(t,\o) \in [0,T]\times\O$.

\begin{theorem}[\cite{PR07}, Theorem 4.2.4]\label{thm:var_ex}
  Assume $(H1)$--$(H4)$. Then, for every $X_0 \in L^2(\O,\mcG_0;H)$ \eqref{eqn:spde3} has a unique solution $X$ satisfying
    \[ \E \left( \sup_{t \in [0,T]}\|X_t\|_H^2 + \int_0^T \|X_r\|_V^\a dr \right) < \infty.\]
\end{theorem}


\bibliographystyle{plain}
\bibliography{refs}   

\end{document}